\documentclass[final]{siamltex}
\usepackage{graphicx}
\usepackage{mathptmx}
\usepackage{amssymb,amsmath, amsfonts}
\usepackage{color}

\newcommand{\email}[1]{{\tt #1}}
\def\tto{\;{\lower 1pt \hbox{$\rightarrow$}}\kern -10pt
\hbox{\raise 2pt \hbox{$\rightarrow$}}\;}
\newcommand{\R}{\mathbb{R}}
\newcommand{\N}{\mathbb{N}}
\newcommand{\norm}[1]{\|#1\|}

\newcommand{\B}{{\cal B}}

\newcommand{\setto}[1]{\mathop{\rightarrow}\limits^#1}
\newcommand{\longsetto}[1]{\mathop{\longrightarrow}\limits^#1}
\newcommand{\skalp}[1]{\langle #1\rangle}
\newcommand{\xb}{\bar x}
\newcommand{\yb}{\bar y}
\newcommand{\zb}{\bar z}
\newcommand{\pb}{\bar p}

\def\O{\Omega}
\def\dn{\downarrow}
\newcommand{\oo}{o}
\def\ox{\bar{x}}
\def\op{\bar{p}}
\def\oy{\bar{y}}
\def\oz{\bar{z}}

\def\gph{\mbox{\rm gph}\,}
\def\kk{\kappa}
\def\disp{\displaystyle}

\newcommand{\I}{{\cal I}}

\newcommand{\Gr}{{\rm gph\,}}

\newcommand{\ph}{\varphi}

\newcommand{\subreg}{{\rm subreg\,}}
\newcommand{\subregx}{{\rm rob\,}}

\newcommand{\bmpx}{{\rm bmp}_x\,}
\newcommand{\ochi}{{\overline\chi}}
\newcommand{\uchi}{{\underline\chi}}
\newcommand{\gb}{g(\pb,\xb)}

\newcommand{\gxb}{\nabla_x g(\pb,\xb)}
\newcommand{\gpb}{\nabla_p g(\pb,\xb)}
\newcommand{\ImDp}{{\rm Im}_\zeta D_pg(\pb,\xb)}
\def\R{\mathbb R}
\def\N{\mathbb N}

\def\ph{\varphi}
\def\emp{\emptyset}

\def\lm{\lambda}

\def\kk{\kappa}

\def\la{\langle}
\def\ra{\rangle}
\def\hat{\widehat}
\def\Tilde{\widetilde}
\def\tilde{\widetilde}
\def\Bar{\overline}
\newlength{\myparboxwidth}

\title{Robinson Stability of Parametric Constraint Systems\\via Variational Analysis}
\author{HELMUT GFRERER\footnote{Institute of Computational Mathematics, Johannes Kepler University Linz, A-4040 Linz, Austria; \email{helmut.gfrerer@jku.at}}
$\;$ \and $\;$ BORIS S. MORDUKHOVICH\footnote{Department of Mathematics, Wayne State University, Detroit, MI 48202, USA, and RUDN University, Moscow 117198, Russia; \email{boris@math.wayne.edu}}}
\begin{document}
\newtheorem{Theorem}{Theorem}[section]
\newtheorem{Proposition}[Theorem]{Proposition}
\newtheorem{Remark}[Theorem]{Remark}
\newtheorem{Lemma}[Theorem]{Lemma}
\newtheorem{Corollary}[Theorem]{Corollary}
\newtheorem{Definition}[Theorem]{Definition}
\newtheorem{Example}[Theorem]{Example}
\newtheorem{Assumption}[Theorem]{Assumption}
\renewcommand{\theequation}{{\thesection}.\arabic{equation}}
\maketitle

{\bf Abstract.} This paper investigates a well-posedness property of parametric constraint systems named here {\em Robinson stability}. Based on advanced tools of variational analysis and generalized differentiation, we derive first-order and second-order conditions for this property under minimal constraint qualifications and establish relationships of Robinson stability with other well-posedness properties in variational analysis and optimization. The results obtained are applied to robust Lipschitzian stability of parametric variational systems.


\section{Introduction and Discussion}

The main focus of this paper is on studying {\em parametric constraint systems} (PCS) of the type
\begin{equation}\label{EqConstrSystem}
g(p,x)\in C\;\mbox{ with }\;x\in\R^n\;\mbox{ and }\;p\in P,
\end{equation}
where $x$ is the {\em decision} variable, and where $p$ is the {\em perturbation parameter} belonging to a topological space $P$. In what follows we impose standard smoothness assumptions on $g\colon P\times\R^n\to\R^l$ with respect to the decision variable and consider general constraint sets $C\subset\R^l$, which are closed while not necessarily convex. Define the (set-valued) {\em solution map} $\Gamma\colon P\tto\R^n$ to \eqref{EqConstrSystem} by
\begin{eqnarray}\label{Gamma}
\Gamma(p):=\big\{x\in\R^n\big|\;g(p,x)\in C\big\}\;\mbox{ for all }\;p\in P
\end{eqnarray}
and fix the reference feasible pair $(\op,\ox)\in\gph\Gamma$. The major attention below is paid to the following well-posedness property of PCS, which postulates the desired local behavior of the solution map \eqref{Gamma}.

\begin{Definition}[Robinson stability]\label{DefUMSP} We say that PCS \eqref{EqConstrSystem} enjoys the {\sc Robinson stability (RS)} property at $(\op,\ox)$ with modulus $\kk\ge0$ if there are neighborhoods $U$ of $\xb$ and $V$ of $\pb$ such that
\begin{equation}\label{EqUMSP}
{\rm dist}\big(x;\Gamma(p)\big)\le\kappa\,{\rm dist}\big(g(p,x);C\big)\;\mbox{ for all }\;(p,x)\in V\times U
\end{equation}
in terms of the usual point-to-set distance. The infimum over all such moduli $\kappa$ is called the {\sc RS exact bound} of \eqref{EqConstrSystem} at $(\op,\ox)$ and is denoted by $\subregx(g,C)(\pb,\xb)$.
\end{Definition}

Robinson \cite{Rob76} studied this property for (closed) {\em convex cones} $C$ under the name of ``stability" and proved that the following condition (known now as the {\em Robinson constraint qualification}):
\begin{eqnarray}\label{rcq}
0\in{\rm int}\big(g(\op,\ox)+\nabla_x g(\op,\ox)\R^n-C\big)
\end{eqnarray}
is {\em sufficient} for RS in this case. Also, \eqref{rcq} is shown to be {\em necessary} for \eqref{EqUMSP} if $g(p,x)=g(x)-p$ (the case of {\em canonical} perturbations) and $P=\R^l$ (or $P$ is a neighborhood of $0\in\R^l$). Further results in this direction have been obtained in various publications (see, e.g., \cite{BonSh00,Bor86,ChiYaoYen10,DoRo14,JeyYen04} and the references therein), and in some of them condition \eqref{EqUMSP} is called ``Robinson metric regularity" of \eqref{Gamma}. In our opinion, the latter name is misleading since it contradicts the widely accepted notion of {\em metric regularity} in variational analysis \cite{Mor06,RoWe98} meaning, for a given set-valued mapping $F\colon Z\tto Y$ between metric spaces and a given point $(\oz,\oy)\in\gph F$, that the following distance estimate
\begin{eqnarray}\label{metreg}
{\rm dist}\big(z;F^{-1}(y)\big)\le\kk\,{\rm dist}\big(y;F(z)\big)\;\mbox{ for all }\;(z,y)\;\mbox{ close to }\;(\oz,\oy)
\end{eqnarray}
holds. Having in mind the weaker property of {\em metric subregularity} of $F$ at $(\oz,\oy)$, which corresponds to the validity of \eqref{metreg} with the fixed point $y=\oy$ therein, we can interpret the RS property \eqref{EqUMSP} as the metric subregularity of the other mapping $x\mapsto g(p,x)-C$ at $(x,0)$ for every point $x\in\Gamma(p)$ close to $\ox$ for every fixed parameter $p\in P$ close to $\op$ with the {\em uniform} modulus $\kk$.

Another useful interpretation of \eqref{EqUMSP} is as follows. Robinson defined in \cite{Rob76} the class of {\em admissible perturbations} of the system  $g(x)\in C$ at $\ox$ as triples $(P,\op,g(p,x))$ such that $\op\in P$ and $g\colon P\times\R^n\to\R^l$ is partially differentiable with respect to $x$ for all $p\in P$, is continuous together with $\nabla_x g$ at $(\op,\ox)$, and satisfies $g(\op,x)=g(x)$ near $\ox$. It can be distilled from \cite{Rob76} that, in the case of convex cones $C$, the metric regularity of the mapping $x\to g(x)-C$ around $(\ox,0)$ is {\em equivalent} to the validity of \eqref{EqUMSP} for {\em all} the admissible perturbations with some uniform modulus $\kk$. However, the situation changes dramatically when we face realistic models with {\em constraints} on feasible perturbations. In such settings, which particularly include canonical perturbations with convex cones $C$ while $\op=0\in{\rm bd}\,P\subset\R^l$, the {\em uniform subregularity} viewpoint on Robinson stability is definitely useful. This approach naturally relates to a challenging issue of variational analysis on determining classes of perturbations under which the (generally {\em nonrobust}) property of metric {\em sub}regularity is {\em stable}. Such developments are important for various applications; see, e.g., \cite{knt10}.

The major goal of this paper is to obtain verifiable conditions on perturbation triples $(P,\op,g(p,x))$ ensuring the validity of the RS estimate \eqref{EqUMSP}. The results obtained in this vein seem to be new not only for the case of general perturbations with nonconvex sets $C$, but even in the conventional settings where perturbations are canonical and $C$ is a polyhedral convex cone. To achieve these results, we use powerful tools of first-order and second-order variational analysis and generalized differentiation, which are briefly reviewed in Section~2. The rest of the paper is organized as follows.

Section~3 presents {\em first-order} results on the validity of Robinson stability and its relationships with some first-order constraint qualifications and Lagrange multipliers. In particular, a precise formula for calculating the {\em exact stability bound} $\subregx(g,C)(\op,\ox)$ is derived under a new {\em subamenability} property of $C$. The main first-order conditions ensuring RS go {\em far beyond metric regularity} of $x\mapsto g(\op,x)-C$ while surely hold under its validity regardless of the convexity of the set $C$. We further specify the obtained results in the settings where $C$ is either convex or the union of finitely many convex polyhedra and also under more conventional constraint qualifications.

Section~4 is devoted to {\em second-order} analysis of Robinson stability, which seems has never been previously done in the literature in the framework of Definition~\ref{DefUMSP}. However, such an analysis of some other stability and regularity properties in the convex constraint framework of Banach spaces under the failure of Robinson's constraint qualification \eqref{rcq} has been efficiently conducted by Arutyunov and his collaborators; see, e.g., \cite{a10,aai07} and the references therein. We introduce here new second-order quantities for closed sets and employ them to derive constructive second-order conditions to ensure Robinson stability of \eqref{EqConstrSystem} in the case of general sets $C$ with effective specifications for unions of convex polyhedra. As a by-product of the obtained results on Robinson stability, new second-order conditions for metric subregularity of constraint mappings are also derived in nonpolyhedral settings.

Section~5 provides applications of the main results on Robinson stability to establish new first-order and second-order conditions for {\em robust Lipschitzian stability} (Lipschitz-like or Aubin property) of solution maps in \eqref{Gamma} with their specifications for {\em parametric variational systems} (PVS). The latter systems reduce to PCS \eqref{EqConstrSystem} with sets $C$ represented as graphs of normal cone/subdifferential mappings (in particular, parameter-dependent ones), which occur to be the most challenging for sensitivity analysis. The given numerical example shows that our results can be efficiently applied to such cases.

In the concluding Section~6 we briefly summarize the obtained results for Robinson stability of PCS, present more discussions on its relationships with other well-posedness properties of PCS and PVS, and outline some topics for our future research.

Throughout the paper we use standard notation of variational analysis and generalized differentiation (see, e.g., \cite{Mor06,RoWe98}), except special symbols discussed in the text. 

\section{Preliminaries from Variational Analysis}

All the sets under consideration are supposed to be {\em locally closed} around the points in question without further mentioning. Given $\Omega\subset\R^d$ and $\oz\in\O$, recall first the standard constructions of variational analysis used in what follows (see \cite{Mor06,RoWe98}):

The (Bouligand-Severi) {\em contingent cone} to $\O$ at $\oz$ is:
\begin{eqnarray}\label{tan}
T_\O(\oz):=\big\{u\in\R^d\big|\;\exists\,t_k\dn 0,\;u_k\to u\;\mbox{ with }\;\oz+t_k u_k\in\O\;\mbox{ for all }\;k\in\N\big\}.
\end{eqnarray}

The (Fr\'echet) {\em regular normal cone} to $\O$ at $\oz$ is:
\begin{eqnarray}\label{rnc}
\Hat N_\O(\oz):=T_\O(\oz)^*=\big\{v\in\R^d\big|\;\la v,u\ra\le 0\;\mbox{ for all }\;u\in T_\O(\oz)\big\}.
\end{eqnarray}

The (Mordukhovich) {\em limiting normal cone} to $\O$ at $\oz$ is:
\begin{eqnarray}\label{lnc}
N_\O(\oz):=\big\{v\in\R^d\big|\;\exists\,z_k\to\ox,\;v_k\to v\;\mbox{ with }\;v_k\in\Hat N_\O(z_k),\;z_k\in\O\;\mbox{ for all }\;k\in\N\big\}.
\end{eqnarray}

We will also employ the directional modification of \eqref{lnc} introduced recently by Gfrerer \cite{Gfr13b}. Given $w\in\R^d$, the {\em limiting normal cone in direction} $w$ to $\O$ at $\oz$ is
\begin{eqnarray}\label{dnc}
N_\O(\oz;w):=\big\{v\in\R^d\big|\;\exists\,t_k\dn 0,\;w_k\to w,\;v_k\to v\;\mbox{ with }\;v_k\in\Hat N_\O(\oz+t_k w_k),\;\oz+t_k w_k\in\O\big\}.\qquad
\end{eqnarray}

The following calculus rule is largely used in the paper. It is an extension of the well-known result of variational analysis (see, e.g., \cite[Theorem~3.8]{Mor06}) with replacing the metric regularity qualification condition by that of the imposed metric subregularity. Note that a similar result in somewhat different framework can be distilled from the proof of \cite[Theorem~4.1]{HenJouOut02}; cf.\ also \cite[Rule ($S_2$)]{io08}.

\begin{Lemma}[limiting normals to inverse images]\label{CorBMP} Let $f:\R^s\to\R^d$ be strictly differentiable at $\oz\in f^{-1}(C)$ and such that the mapping $z\mapsto f(z)-C$ is metrically subregular at $(\oz,0)$ with modulus $\kappa\ge 0$. Then for every $v\in N_{f^{-1}(C)}(\oz)$ there exists some $u\in N_C(f(\oz))\cap\kappa\|v\|\B_{\R^d}$ satisfying $v=\nabla f(\oz)^*u$, where the sign $^*$ indicates the matrix transposition.
\end{Lemma}
\begin{proof} Denote $F(z):=f(z)-C$ and pick $v\in N_{f^{-1}(C)}(\oz)=N_{F^{-1}(0)}(\oz)$. Since the mapping $F:\R^s\tto\R^d$ is metrically subregular at $(\oz,0)$ with modulus $\kappa$ and closed-graph around this point, we can apply \cite[Proposition~4.1]{GfrOut15} and get the inclusion
\begin{eqnarray}\label{inv-im}
N_{F^{-1}(\yb)}(\oz;w)\subset\big\{v\in\R^s\big|\;\exists\,y\in\kappa\|v\|\B_{\R^d}\;\mbox{ with }\;(v,y)\in N_{\Gr F}\big((\oz,0);(w,0)\big)\big\}.
\end{eqnarray}
Using \eqref{inv-im} with $w=0$ yields the existence of $u\in\kappa\|v\|\B_{\R^l}$ such that $(v,-u)\in N_{\Gr F}(\oz,0)$. The structure of the mapping $F$ and elementary differentiation ensure the normal cone representation
$$
N_{\Gr F}(\oz,0)=\big\{(\xi,\eta)\in\R^s\times\R^d\big|\;-\eta\in N_C\big(f(\oz)\big),\;\xi+\nabla f(\oz)^*\eta=0\big\},
$$
which therefore verifies the claimed statement of the lemma.\end{proof}

Throughout the paper we systematically distinguish between metric regularity and subregularity assumptions. To illuminate the difference between these properties in the case of the underlying mapping $F(\cdot)=f(\cdot)-C$, observe that $F$ is metrically regular around $(\oz,0)$ {\em if and only if} the implication
\begin{eqnarray}\label{mr-f}
\big[\lm\in N_C\big(f(\oz)\big),\;\nabla f(\oz)^*\lm=0\big]\Longrightarrow\lm=0
\end{eqnarray}
holds. This is a direct consequence of the Mordukhovich criterion; see \cite[Theorem~9.40]{RoWe98}. On the other hand, it is shown in \cite[Corollary~1]{GfrKl15} based on the results developed in \cite{Gfr11} that the metric subregularity of $F$ at $(\oz,0)$ is guaranteed by the condition that for all $u\ne0$ with $\nabla f(\oz)u\in T_C(f(\ox))$ we have the implication
\begin{eqnarray}\label{msr-f}
\big[\lm\in N_C\big(f(\oz);\nabla f(\oz)u\big),\;\nabla f(\oz)^*\lm=0\big]\Longrightarrow\lm=0.
\end{eqnarray}

Next we introduce a new class of ``nice" sets the properties of which extend the corresponding ones for {\em amenable} sets; see \cite{RoWe98}. The difference is again in employing metric subregularity instead of metric regularity. Indeed, the qualification condition used in the definition of amenability \cite[Definition~10.23]{RoWe98} ensures the metric regularity of the mapping $z\mapsto q(z)-Q$ below around $(\oz,0)$.

\begin{Definition}[subamenable sets]\label{subamen} A set $C\subset\R^l$ is called {\sc subamenable} at $\oz\in C$ if there exists a neighborhood $W$ of $\bar z$ along with a ${\cal C}^1$-smooth mapping $q\colon W\to\R^d$ for some $d\in\N$ and along with a closed convex set $Q\subset\R^d$ such that we have the representation
\begin{eqnarray*}
C\cap W=\big\{z\in W\big|\;q(z)\in Q\big\}
\end{eqnarray*}
and the mapping $z\mapsto q(z)-Q$ is metrically subregular at $(\bar z,0)$. We say that $C$ is {\sc strongly subamenable} at $(\oz,0)$ if this can be arranged with $q$ of class ${\cal C}^2$, and it is {\sc fully subamenable} at $(\oz,0)$ if in addition the set $Q$ can be chosen as a convex polyhedron.
\end{Definition}

Finally in this section, we formulate our {\em standing assumptions} on the mapping $g\colon P\times\R^n\to\R^l$ in \eqref{EqConstrSystem}, which stay without further mentioning for the rest of the paper: There are neighborhoods $U$ of $\xb$ and $V$ of $\pb$ such that for each $p\in V$ the mapping $g(p,\cdot)$ is continuously differentiable on $U$ and that both $g$ and and its partial derivative $\nabla_x g$ are continuous at $(\pb,\xb)$.

\section{First-Order Analysis of Robinson Stability}

We start this section with establishing relationships between Robinson stability and other important properties and constraint qualifications for parametric systems \eqref{EqConstrSystem}. Recall first from \cite{GM15} that the partial {\em metric subregularity constraint qualification} (MSCQ) holds for \eqref{EqConstrSystem} at $(\op,\ox)$ with respect to $x$ if there are neighborhoods $U$ of $\xb$ and $V$ of $\pb$ such that for every $p\in V$ and every $x\in\Gamma(p)\cap U$ the mapping $g(p,\cdot)-C$ is metrically subregular at $(x,0)$, i.e., there is a neighborhood $U_{p,x}$ of $x$ and a constant $\kappa_{p,x}\ge 0$, possibly depending on $p$ and $x$, for which
\begin{eqnarray}\label{mscq}
{\rm dist}\big(\xi;\Gamma(p)\big)\le\kappa_{p,x}{\rm dist}\big(g(p,\xi);C\big)\;\mbox{ for all }\;\xi\in U_{p,x}.
\end{eqnarray}

Next we introduce a new property of PCS \eqref{EqConstrSystem} at the reference point $(\op,\ox)$ that involves limiting  normals and Lagrange multipliers. Given $(p,x)\in\Gr\Gamma$ and $v\in\R^n$, define the set of multipliers
\begin{eqnarray*}
\Lambda(p,x,v):=\big\{\lambda\in N_C\big(g(p,x)\big)\big|\;\nabla_x g(p,x)^*\lambda=v\big\}.
\end{eqnarray*}
It follows from construction of $\Gamma$ in \eqref{Gamma} and Lemma~\ref{CorBMP} that the metric subregularity of the mapping $g(p,\cdot)-C$ at $(x,0)$ ensures the normal cone representation
$N_{\Gamma(p)}(x)\subset\nabla_x g(p,x)^*N_C(g(p,x))$, and therefore the inclusion $v\in N_{\Gamma(p)}(x)\Longrightarrow\Lambda(p,x,v)\ne\emp$ holds.

\begin{Definition}[partial bounded multiplier property]\label{DefBMP}  We say that the {\sc partial bounded multiplier property} $($BMP$)$ with respect to $x$ is satisfied for system  \eqref{EqConstrSystem} at the point $(\pb,\xb)\in\gph\Gamma$ with modulus $\kappa\ge0$ if there are neighborhoods $V$ of $\pb$ and $U$ of $\xb$ such that
\begin{eqnarray}\label{EqPBMP}
\Lambda(p,x,v)\cap\kappa\|v\|\B_{\R^n}\ne\emp\;\mbox{ for all }\;p\in V,\;x\in\Gamma(p)\cap U,\;\mbox{ and }\;v\in N_{\Gamma(p)}(x).
\end{eqnarray}
The infimum over all such moduli $\kappa$ is denoted by $\bmpx(g,C)(\pb,\xb)$.
\end{Definition}

To emphasize what is behind Robinson stability, consider an important particular case of constraint systems in {\em nonlinear programming} (NLP) described by smooth equalities and inequalities
\begin{eqnarray*}
g_i(p,x)=0\;\mbox{ for }\;i=1,\ldots,l_E\;\mbox{ and }\;g_i(p,x)\le 0\;\mbox{ for }\;i=l_E+1,\ldots,\;l_E+l_I=l.
\end{eqnarray*}
Such systems can be represented in the form of \eqref{EqConstrSystem} as follows:
\begin{eqnarray}\label{EqEqualInequal}
g(p,x)\in C:=\{0\}^{l_E}\times\R_-^{l_I}.
\end{eqnarray}

It has been well recognized in nonlinear programming that, given $(\op,\ox)\in\gph\Gamma$, the metric regularity of the mapping $g(\op,\cdot)-C$ around $(\ox,0)$ in the setting of \eqref{EqEqualInequal} is equivalent to the (partial) {\em Mangasarian-Fromovitz constraint qualification} (MFCQ) with respect to $x$ at $(\op,\ox)$, which in turn ensures the uniform boundedness of Lagrange multipliers around this point. Then the robustness of metric regularity allows us to conclude that the partial MFCQ implies the validity of both partial MSCQ and BMP for \eqref{EqEqualInequal} at $(\op,\ox)$ with some modulus $\kappa>0$.

Let us now recall another classical constraint qualification ensuring both partial MSCQ and BMP for \eqref{EqEqualInequal}. Denote $E:=\{1,\ldots, l_E\}$, $I:=\{l_I+1,\ldots,l\}$ and for any $(p,x)$ feasible to \eqref{EqEqualInequal} consider the index set $\I(p,x):=\{i\in I|\;g_i(p,x)=0\}$ of active inequalities and then put $I^+(\lambda):=\{i\in I|\; \lambda_i>0\}$ where $\lambda=(\lambda_1,\ldots,\lm_l)$. It is said that the partial {\em constant rank constraint qualification} (CRCQ) with respect to $x$ holds at $(\pb,\xb)\in\gph\Gamma$ if there are neighborhoods $V$ of $\pb$ and $U$ of $\xb$ such that for every subset $J\subset E\cup\I(\pb,\xb)$ the family of partial gradients $\{\nabla_xg_i(p,x)|\;i\in J\}$ has the same rank on $V\times U$.

\begin{Proposition}[MSCQ and BMP follow from CRCQ]\label{PropCRCQ_BMP} Given $(\pb,\xb)\in\Gr\Gamma$ for \eqref{EqEqualInequal}, the partial CRCQ at $(\op,\ox)$ implies that both partial MSCQ and BMP with respect to $x$ hold at this point.
\end{Proposition}
\begin{proof} Implication CRCQ$\Longrightarrow$MSCQ for the partial versions under consideration can be deduced from \cite[Proposition~2.5]{Jan84}. Let us verify that CRCQ$\Longrightarrow$BMP. Assuming the contrary, find sequences $(p_k,x_k)\longsetto{{\Gr\Gamma}}(\pb,\xb)$ and $v_k\in\B_{\R^n}\cap\nabla g_x(p_k,x_k)^*N_C(g(p_k,x_k))$ for which $v_k\not\in\nabla_x g(p_k,x_k)^*(N_C(g(p_k,x_k))\cap k\B_{R_n}$ as $k\in\N$. Choose a subset $E'\subset E$ such that $\{\nabla_x g_i(\pb,\xb)\|\;i\in E'\}$ is a base of the span of the gradient family $\{\nabla_x g_i(\pb,\xb)|\;i\in E\}$. The imposed partial CRCQ tells us that for all $k$ sufficiently large the set $\{\nabla_x g_i(p_k,x_k)|\;i\in E'\}$ is also a base of the span of $\{\nabla_x g_i(p_k,x_k)|\;i\in E\}$, and hence the set
$$
\Lambda_k:=\big\{\lambda\in N_C\big(g(p_k,x_k)\big)\big|\;v_k=\nabla_x g(p_k,x_k)^*\lambda,\;\lambda_i=0\;\mbox{ for }\;i\in E\setminus E'\big\}
$$
is a nonempty convex polyhedron having at least one extreme point. Let $\lambda^k$ denote an extreme point of $\Lambda_k$ meaning that the gradient family $\{\nabla_x g_i(p_k,x_k)|\;i\in E'\cup I^+(\lambda^k)\}$ is linearly independent. After passing to a subsequence if needed, suppose that the sequence $\lambda^k/\norm{\lambda^k}$ converges to some $\lambda\in N_C(g(\pb,\xb))\cap\B_{\R^l}$. Since $\|\lambda^k\|>k$ and $v_k\in\B_{\R^n}$, we obtain
\begin{eqnarray*}
\Big\|\nabla_xg(p_k,x_k)^*\frac{\lambda^k}{\|\lambda^k\|}\Big\|=\frac{\norm{x_k^\ast}}{\norm{\lambda^k}}<\frac 1k,\quad k\in\N,
\end{eqnarray*}
yielding $\nabla_x g(\pb,\xb)^*\lambda=0$. Since $\lambda^k_i=0$ for $i\in E\setminus E'$, we also have $\lambda_i=0$ for these indices, and so the family $\{\nabla_x g_i(\pb,\xb)|\; i\in E'\cup I^+(\lambda)\}$ is linearly dependent. By $I^+(\lambda)\subset\I(\pb,\xb)$ due to $\lambda\in N_C(g(\pb,\xb))$ and by $\lambda^k_i>0$ for each $i\in I^+(\lambda)$ when $k\in$ is large, it follows that $I^+(\lambda)\subset I^+(\lambda^k)$. The partial CRCQ with respect to $x$ at $(\pb,\xb)$ ensures that the family $\{\nabla_x g_i(p_k,x_k)|\; i\in E'\cup I^+(\lambda)\}$ is linearly dependent, and hence the family $\{\nabla_x g_i(p_k,x_k)|\;i\in E'\cup I^+(\lambda^k)\}$ is linearly dependent as well. This contradicts our choice of $\lambda^k$ and thus shows that the partial BMP with respect to $x$ must hold at $(\pb,\xb)$.\end{proof}

Following \cite{JeyYen04} and slightly adjusting the name, we say that $\xb$ is a {\em parametrically stable} solution to \eqref{EqConstrSystem} on $P$ at $\pb$ if for every neighborhood $U$ of $\xb$ there is some neighborhood $V$ of $\pb$ such that
\begin{eqnarray}\label{stab-sol}
\Gamma(p)\cap U\ne\emp\;\mbox{ whenever }\;p\in V.
\end{eqnarray}

Now we are ready to establish relationships between Robinson stability of PCS \eqref{EqConstrSystem} and the aforementioned properties and constraint qualifications.

\begin{Theorem}[first-order relationships for Robinson stability]\label{ThEquMSP_BMP} Let $(\op,\ox)\in\gph\Gamma$ in \eqref{Gamma}, and let $\kk\ge 0$. Consider the following statements for PCS \eqref{EqConstrSystem} under the imposed standing assumptions:

{\bf(i)} Robinson stability holds for \eqref{EqConstrSystem} at $(\op,\ox)$.

{\bf(ii)} The point $\xb$ is a parametrically stable solution to \eqref{EqConstrSystem} on $P$ at $\pb$, and the partial MSCQ together with the partial BMP with respect to $x$ are satisfied at $(\op,\ox)$.

Then {\rm (i)}$\Longrightarrow${\rm(ii)} with $\bmpx(g,C)(\pb,\xb)\le\subregx(g,C)(\pb,\xb)$. Conversely, if $C$ is subamenable at $g(\pb,\xb)$, then {\rm (ii)}$\Longrightarrow${\rm (i)} and we have the exact formula $\subregx(g,C)(\pb,\xb)=\bmpx(g,C)(\pb,\xb)$.
\end{Theorem}

\begin{proof}
To verify (i)$\Longrightarrow$(ii), find by (i) neighborhoods $V$ and $U$ such that the standing assumptions and the estimate \eqref{EqUMSP} are satisfied with some modulus $\kk\ge 0$. Then the partial MSCQ with respect to $x$ follows directly from definition \eqref{mscq}. Furthermore, for every $(p,x)\in(V\times U)\cap\Gr\Gamma$ the mapping $M_p\colon\R^n\rightrightarrows\R^l$ with $M_p(\xi):=g(p,\xi)-C$ is metrically subregular at $(x,0)$ with the same modulus $\kk$, and thus Lemma~\ref{CorBMP} tells us that for each $(p,x)\in\Gr\Gamma\cap(V\times U)$ and each $v\in N_{\Gamma(p)}(x)= N_{M_p^{-1}(0)}(x)$ there exists some $\lambda\in\|v\|\kappa\B_{\R^l}\cap N_C(g(p,x))$ with $v=\nabla_xg(p,x)^*\lambda$. This justifies \eqref{EqPBMP} and the partial BMP with respect to $x$ at $(\pb,\xb)$, and therefore $\bmpx(g,C)(\pb,\xb)\le\subregx(g,C)(\pb,\xb)$.

To finish the proof of the claimed implication, it remains to show the parametric stability of $\ox$. Take any neighborhood $\Tilde U$ of $\xb$ and find the radius $r>0$ with ${\rm int}\B(\xb;r)\subset\Tilde U$. By the continuity of $g$ there is a neighborhood $\Tilde V$ of $\pb$ such that $\kappa\|g(p,\xb)-g(\pb,\xb)\|< r$ whenever $p\in\tilde V$. This yields
\begin{eqnarray*}
{\rm dist}\big(\xb;\Gamma(p)\big)\le\kappa{\rm dist}\big(g(p,\xb);C\big)\le\kappa\|g(p,\xb)-g(\pb,\xb)\|< r\;\mbox{ for all }\;p\in V\cap\Tilde V,
\end{eqnarray*}
which shows therefore that $\emp\ne\Gamma(p)\cap{\rm int}\B(\xb;r)\subset\Gamma(p)\cap\Tilde U$ for all $p\in V\cap\Tilde V$. Hence $\xb$ is parametrically stable on $P$ at $\pb$ by definition \eqref{stab-sol}, and we fully justify implication (i)$\Longrightarrow$(ii).

To verify the converse implication (ii)$\Longrightarrow$(i), suppose that $C$ is subamenable at $\gb$. By (ii) take $q$, $Q$, and $W$ according to Definition~\ref{subamen} and find neighborhoods $V$ of $\pb$ and $U$ of $\xb$ for which conditions \eqref{mscq}, \eqref{EqPBMP}, and the standing assumptions are satisfied. Choosing $\kappa>0$ so that \eqref{EqPBMP} holds and by shrinking $W$ if necessary, suppose that $q(\cdot)-Q$ is metrically subregular with modulus $\kappa_C$ at every point $(z,0)$ with $z\in q^{-1}(Q)\cap W$. Denote
\begin{eqnarray*} \label{EqConstantL}
L:=2\kappa_C(3+\norm{\gxb})+2,
\end{eqnarray*}
and let $0<\delta<\min\{1/\kappa L,1\}$ be arbitrarily fixed. Then choose some constant $r_z>0$ so that
$
{\rm int}\B\big(\gb;2r_z\big)\subset W\;\mbox{ and }\;\norm{q(z)-q(z')-\nabla q(z')(z-z')}\le\delta\norm{z-z'}\;\mbox{ for all }\;z,z'\in{\rm int}\B\big(\gb;2r_z\big).
$
Shrinking $U$ and $V$ allows us to get $\norm{\nabla_x g(p,x)-\nabla_x g(\pb,\xb)}\le\delta$ and $\norm{g(p,x)-\gb}<r_z$ for all $(p,x)\in V\times U$. Further, let $r>0$ be such that $\B(\xb;3r)\subset U$ and, by the parametric stability of $\xb$ on $P$ and by shrinking again $V$, we have $\Gamma(p)\cap{\rm int}\B(\xb;r)\ne\emp$ whenever $p\in V$. Fix now $(x,p)\in{\rm int}\B(\xb;r)\times V$ and let $\bar\xi$ be a global solution to the {\em optimization} problem
\begin{eqnarray}\label{opt}
\mbox{ minimize }\;\frac 12\norm{\xi-x}^2\;\mbox{ subject to }\;\xi\in\Gamma(p).
\end{eqnarray}
Such a global solution surely exists due to the closedness of $\Gamma(p)\cap\B(\xb;3r)$. Then $\norm{\bar \xi-x}<2r$ and hence $\norm{\bar\xi-\xb}<3r$ yielding $\bar\xi\in{\rm int}\B(\xb;3r)\subset U$. This verifies the metric subregularity of $g(p,\cdot)-C$ at $(\bar\xi,0)$. Applying now the necessary optimality condition in \eqref{opt} from \cite[Proposition~5.1]{Mor06a} and then using Lemma~\ref{CorBMP} give us the inclusions
$$
-(\bar\xi-x)\in\Hat N_{\Gamma(p)}(\bar\xi)\subset\nabla_x g(p,\bar\xi)^*N_C\big(g(p,\bar\xi)\big),
$$
which show that $\Lambda(p,\bar\xi,x-\bar\xi)\ne\emp$, and thus there is a multiplier $\lambda\in\Lambda(p,\bar\xi,x-\bar\xi)\cap\kappa\norm{\bar\xi-x}\B_{\R^l}$ by the imposed partial BMP. Moreover, we have
$$
g(p,\bar\xi)\in{\rm int}\B\big(\gb;r_z\big)\cap C\subset W\cap q^{-1}(Q)
$$
while concluding therefore that the mapping $q(\cdot)-Q$ is metrically subregular at $(g(p,\bar\xi),0)$. Using again Lemma~\ref{CorBMP} provides a multiplier $\mu\in N_Q(q(g(p,\bar\xi))\cap\norm{\lambda}\kappa_C\B_{\R^{d}}$ with $\lambda=\nabla q(g(p,\bar\xi))^*\mu$.

Choose now $\bar z$ as a projection of the point $g(p,x)$ on the set $C$. Since $g(p,x)\in{\rm int}\B(\gb;r_z)$ and $\gb\in C$, we get $\norm{\bar c-\gb}<2r_z$ and thus
\begin{eqnarray*}
\skalp{\lambda,\bar z-g(p,\bar\xi)}=\skalp{\mu,\nabla q(g(p,\bar\xi))(\bar z-g(p,\bar\xi))}\le\skalp{\mu,q(\bar z)-q(g(p,\bar\xi))}+\delta\norm{\mu}\norm{\bar z-g(p,\bar\xi)}.
\end{eqnarray*}
Using $\skalp{\mu,q(\bar z)-q(g(p,\bar\xi))}\le 0$ by $\mu\in N_Q(q(g(p,\bar\xi)))$ together with $\norm{\bar z-g(p,\bar\xi)}\le\norm{\bar z-g(p,x)}+\norm{g(p,x)-g(p,\bar\xi)}\le 2\norm{g(p,x)-g(p,\bar\xi)}$ gives us the estimate
\begin{eqnarray*}
\skalp{\lambda,\bar z-g(p,\bar\xi)}\le 2\kappa_C\norm{\lambda}\delta\norm{g(p,x)-g(p,\bar\xi)}.
\end{eqnarray*}
Now taking into account the relationships
\begin{eqnarray*}
\lefteqn{\big\|g(p,x)-g(p,\bar\xi)-\nabla_x g(p,\bar\xi)(x-\bar\xi)\Big\|=\Big\|\int_0^1\Big(\nabla_x g\big(p,\bar\xi+t(x-\bar\xi)\big)-\nabla _x g(p,\bar\xi)\Big)(x-\bar\xi){\rm d}t\Big\|}\\
&\le&\int_0^1\Big(\big\|\nabla_x g\big(p,\bar\xi+t(x-\bar\xi)\big)-\nabla _x g(p,\bar\xi)\big\|\cdot\norm{x-\bar\xi}\Big){\rm d}t\\
&\le&\int_0^1\Big(\norm{\nabla_x g\big(p,\bar\xi+t(x-\bar\xi)\big)-\nabla_x g(\pb,\xb)}+\norm{\nabla _x g(p,\bar\xi)-\nabla_x g(\pb,\xb)}\Big)\norm{x-\bar\xi}{\rm d}t\le2\delta\norm{x-\bar\xi},
\end{eqnarray*}
we obtain $\norm{g(p,x)-g(p,\bar\xi)}\le(2\delta+\norm{\nabla_x g(p,\bar\xi)})\norm{x-\bar\xi}\le(3\delta+\norm{\gxb})\norm{x-\bar\xi}$ and
\begin{eqnarray*}
\norm{x-\bar\xi}^2&=&\skalp{\nabla_x g(p,\bar\xi)^*\lambda,x-\bar\xi}\\
&\le&\skalp{\lambda,g(p,x)-g(p,\bar{\xi})}+\norm{\lambda}\cdot\norm{g(p,x)-g(p,\bar\xi)-\nabla_x g(p,\bar\xi)(x-\bar\xi)}\\
&\le&\skalp{\lambda,g(p,x)-\bar z}+\skalp{\lambda,\bar z-g(p,\bar\xi)}+2\norm{\lambda}\delta\norm{x-\bar\xi}\\
&\le&\norm{\lambda}{\rm dist}\big(g(p,x);C\big)+\norm{\lambda}\delta\big(2\kappa_C\big(3\delta+\norm{\gxb})+2\big)\norm{x-\bar\xi}\\
&\le&\kappa\norm{x-\bar\xi}\big({\rm dist}(g(p,x);C)+\delta L\norm{x-\bar\xi}\big)
\end{eqnarray*}
due to \eqref{EqConstantL} and $\delta<1$. Rearranging yields
\begin{eqnarray*}
(1-\delta\kappa L)\norm{x-\bar\xi}^2\le\kappa\norm{x-\bar\xi}{\rm dist}\big(q(p,x);C\big),
\end{eqnarray*}
which implies that $\norm{x-\bar\xi}={\rm dist}(x;\Gamma(p))\le\frac{\kappa}{1-\delta\kappa L}{\rm dist}(g(p,x);C)$ and that the claimed Robinson stability holds with modulus $\kappa/(1-\delta\kappa L)$. This verifies implication (ii)$\Longrightarrow$(i). Finally, the arbitrary choice of $\delta>0$ close to zero allows us to conclude that the inequality $\subregx(g,C)(\pb,\xb)\le\bmpx(g,C)(\pb,\xb)$ is satisfied. Remembering the opposite inequality derived above, we arrive at the equality $\subregx(g,C)(\pb,\xb)=\bmpx(g,C)(\pb,\xb)$ and thus complete the proof of the theorem.\end{proof}

Note that the results of Theorem~\ref{ThEquMSP_BMP} yield new formulas for calculating the {\em exact bound} (infimum) ``subreg" of subregularity moduli of nonconvex mappings as follows.

\begin{Corollary}[calculating the exact subregularity bound]\label{CorExSubregBnd} Let $f\colon\R^n\to\R^l$ be continuously differentiable, and let $C\subset\R^l$ be subamenable at  $f(\xb)\in C$. If the mapping $x\mapsto F(x):=f(x)-C$ is metrically subregular at $(\xb,0)$, then the exact subregularity bound of $F$ at $(\ox,0)$ is calculated by
\begin{eqnarray}\nonumber
\subreg F(\xb,0)&=&\disp\limsup_{x\longsetto{{f^{-1}(C)}}\xb}\;\disp\sup_{v\in N_{f^{-1}(C)}(x)\cap\B_{\R^n}}\inf\big\{\|\lambda\|\big|\;\lambda\in N_C\big(f(x)\big),\,\nabla f (x)^*\lambda=v\big\}\\
\label{EqSubregBnd}&=&\disp\limsup_{x\longsetto{{f^{-1}(C)}}\xb}\;\inf\big\{\tau\ge 0\big|\;N_{f^{-1}(C)}(x)\cap\B_{\R^n}\subset\tau\nabla f(x)^*\big(N_C\big(f(x)\big)\cap\B_{\R^l}\big)\big\}\qquad\\
\nonumber &=&\disp\limsup_{x\longsetto{{f^{-1}(C)}}\xb}\;\inf\big\{\tau\ge 0\big|\;N_{f^{-1}(C)}(x)\cap\B_{\R^n}\subset\tau D^*F(x,0)(\B_{\R^l})\big\},
\end{eqnarray}
where $D^*F(x,0)(\lm):=\nabla f(x)^*\lm$ if $\lm\in N_C(f(x))$ and $D^*F(x,0)(\lm):=\emp$ otherwise.
\end{Corollary}
\begin{proof} It follows from Theorem~\ref{ThEquMSP_BMP} when $g(p,x):=f(x)$. Indeed, in this case we have the relationships $\subreg F(\ox,0)=\subregx(g,C)(\pb,\xb)=\bmpx(g,C)(\pb,\xb)$, where the latter quantity is calculated by using the normal cone representation for inverse images from Lemma~\ref{CorBMP} with taking into account the imposed subamenability of $C$ and that the metric subregularity of $F$ at $(\ox,0)$ implies this property for $F$ at any $(x,0)$ with $x\in F^{-1}(0)$ close to $\ox$.
\end{proof}

The last formula in \eqref{EqSubregBnd} corresponds to the result by Zheng and Ng \cite[condition~(3.6)]{ZhNg07} obtained for convex-graph multifunctions, which is not the case in Corollary~\ref{CorExSubregBnd}. Observe that the subregularity bound calculations in \eqref{EqSubregBnd} are of a different type in comparison with known modulus estimates for subregularity (see, e.g., Kruger \cite{Kru15a} and the references therein), because \eqref{EqSubregBnd} uses information at points $x\in f^{-1}(C)$ {\em near} $\xb$ while other formulas usually apply quantities at points $x$ {\em outside} the set $f^{-1}(C)$. The main advantage of Corollary~\ref{CorExSubregBnd} in comparison with known results on subregularity for general nonconvex mappings is that we now precisely calculate the {\em exact bound} of subregularity while previous results provided only modulus estimates. It also seems to us that the subregularity modulus estimates of type \cite{Kru15a} are restrictive for applications to Robinson stability interpreted as the {\em uniform} metric subregularity; see Section~1. Indeed, the latter property is {\em robust} for the class of perturbations under consideration while the usual subregularity is not. Since this issue is not detected by estimates of type \cite{Kru15a}, it restricts their ``robust" applications.

The next theorem is the main result of this section providing verifiable conditions for Robinson stability of PCS \eqref{EqConstrSystem} involving the class of perturbation parameters $(P,\op,g(p,x))$ under consideration. For convenience of further applications we split the given system \eqref{EqConstrSystem} into two parts $(i=1,2)$:
\begin{eqnarray}\label{split}
g(p,x)=\big(g_1(p,x),g_2(p,x)\big)\in C_1\times C_2=C,g_i:P\times\R^n\to\R^{l_i},C_i\subset\R^{l_i},l_1+l_2=l
\end{eqnarray}
in such a way that it is known in advance (or easier to determine) that RS holds for $g_2(p,x)\in C_2$, while it is challenging to clarify this for the whole system $g(p,x)\in C$. It is particularly useful for the subsequent second-order analysis of RS and its applications to variational systems; see Sections~4,5.

Since our parameter space $P$ is general topological, we need a suitable differentiability notion for $g(p,x)$ with respect to $p$. It can be done by using the following approximation scheme in the image space $\R^l$. Given any $\zeta:P\to\R$ continuous at $\pb$, define the {\em image derivative} $\ImDp$ of $g$ in $p$ at $(\op,\ox)$ as the closed cone generated by $0$ and those $v\in\R^l$ for which there is a sequence $\{p_k\}\subset P$ with
\begin{eqnarray}\label{Dp}
\begin{array}{c}
\disp 0<\norm{g(p_k,\xb)-g(\pb,\xb)}<k^{-1},\norm{\nabla_x g(p_k,\xb)-\gxb}<k^{-1},\vert\zeta(p_k)-\zeta(\pb)\vert<k^{-1},\\
v=\disp\lim_{k\to\infty}\frac{g(p_k,\xb)-g(\pb,\xb)}{\norm{g(p_k,\xb)-g(\pb,\xb)}}.
\end{array}
\end{eqnarray}
If $P$ is a metric space  with metric $\rho$, the convergence of $p_k\to\pb$ can be ensured by letting $\zeta(p):=\rho(p,\pb)$. If $P$ is a subset of a normed space and $g$ is differentiable with respect to $p$ at $(\pb,\xb)$, then we obviously have the inclusion (with the contingent cone defined in \eqref{tan} via the norm topology of $P$)
\begin{eqnarray}\label{Dp-inj}
\gpb T_P(\pb)\subset\ImDp
\end{eqnarray}
valid for any function $\zeta$. Observe that inclusion \eqref{Dp-inj} holds as equality with $\zeta(p)=\|p-\op\|$ if the operator $\nabla_p g(\op,\ox)$ is injective while the inclusion may be strict otherwise.

\begin{Theorem}[first-order verification of Robinson stability for splitting systems]\label{ThPartFOReg} Given $\zeta:P\to\R$ continuous at $\ox$, assume that for every $v\in\ImDp$ and every $t_k\dn 0$ there is $u\in\R^n$ satisfying
\begin{eqnarray}\label{EqExistenceTangDir}
\disp\liminf_{k\to\infty}\big[{\rm dist}\big(\gb+t_k(v+\gxb u);C\big)/t_k\big]=0.
\end{eqnarray}
Suppose also that Robinson stability at $(\op,\ox)$ holds for the system $g_2(p,x)\in C_2$ in \eqref{split} and that for every $(0,0)\not=(v,u)\in\ImDp\times\R^n$ with $v+\gxb u\in T_C(\gb)$ we have
\begin{eqnarray}\label{EqPartFOReg}
\Big[\lambda=(\lambda^1,\lambda^2)\in\disp\prod_{i=1}^2 N_{C_i}\big(g_i(\pb,\xb);v_i+\nabla_x g_i(\pb,\xb)u\big),\;\nabla_x g(\pb,\xb)^*\lambda=0\Big]\Longrightarrow\lambda^1=0,
\qquad\end{eqnarray}
where $v=(v_1,v_2)\in\R^{l_1}\times\R^{l_2}$. Then Robinson stability at $(\op,\ox)$ holds for the whole system \eqref{split}.
\end{Theorem}
\begin{proof} Assuming on the contrary that Robinson stability fails for \eqref{split} at $(\op,\ox)$, for any $\kappa>0$ and any neighborhoods $U$ of $\xb$ and $V$ of $\pb$ we find $(p,x)\in V\times U$ such that
\begin{eqnarray}\label{EqNonMSP}
{\rm dist}\big(x;\Gamma(p)\big)>\kappa{\rm dist}\big(g(p,x);C\big).
\end{eqnarray}
Our goal is to show by several steps that \eqref{EqNonMSP} eventually contradicts the imposed assumption \eqref{EqPartFOReg} by using first-order necessary optimality conditions in a certain nonsmooth optimization problem under the metric subregularity constraint qualification.

First observe that, since Robinson stability at $(\pb,\xb)$ holds for the system $g_2(p,x)\in C_2$, we get $R>0$ together with a neighborhood $\Bar V$ of $\pb$ and a positive constant $\kappa_2$ satisfying
\begin{eqnarray}\label{EqPartMSP}
{\rm dist}\big(x;\Gamma_2(p)\big)\le\kappa_2{\rm dist}\big(g_2(p,x);C_2\big)\;\mbox{ for all }\;p\in\Bar V,\;x\in{\rm int}\B(\xb;R),
\end{eqnarray}
where $\Gamma_2(p):=\{x|\;g_2(p,x)\in C_2\}$. The standing assumptions allow us to claim that for every $p\in\bar V$ the mapping $g(p,\cdot)$ is continuously differentiable on ${\rm int}\B(\xb;R)$ and then to construct a sequence of neighborhoods $V_k\subset V_{k-1}$ with $V_0:=\bar V$ together with positive radii $R_k\le\min\{R/2,1/k\}$ such that
\begin{eqnarray*}
\norm{\nabla_x g(p,x)-\nabla_x g(\pb,\xb)}\le k^{-1}\;\mbox{ whenever }\;(p,x)\in V_k\times{\rm int}\B(\xb;R_k).
\end{eqnarray*}
Furthermore, for each $k\in\N$ there exist a neighborhood $\Bar V_k\subset V_k$ and a radius $r_k\le R_k/4$ for which
\begin{eqnarray*}
\norm{g(p,x)-g(\pb,\xb)}<\disp\min\Big\{\frac{R_k}{4k^2\max\{\kappa_2,1\}},\frac 1k\Big\}\;\mbox{ as }\;(p,x)\in\Bar V_k\times{\rm int}\B(\xb;r_k).
\end{eqnarray*}
There is no loss of generality to suppose that $\vert\zeta(p)-\zeta(\pb)\vert\le 1/k$ for all $p\in\Bar V_k$. According to \eqref{EqNonMSP} we select $(p_k,x_k)\in\Bar V_k\times{\rm int}\B(\xb;r_k)$ satisfying
\begin{eqnarray*}
{\rm dist}\big(x_k;\Gamma(p_k)\big)>(k+\kappa_2){\rm dist}\big(g(p_k,x_k);C\big)
\end{eqnarray*}
and by using \eqref{EqPartMSP} find $\Tilde x_k\in\Gamma_2(p_k)$ such that
\begin{eqnarray*}
\norm{\Tilde x_k-x_k}\le\kappa_2{\rm dist}\big(g_2(p_k,x_k);C_2\big)\le\kappa_2\norm{g_2(p_k,x_k)-g_2(\pb,\xb)}<\frac{R_k}{4k^2}
\end{eqnarray*}
and hence $\norm{\Tilde x_k-\xb}\le\norm{\Tilde x_k-x_k}+\norm{x_k-\xb}<R_k/4k^2+R_k/4\le R_k/2$. Since $\norm{\nabla_x g(p_k,x)}\le\norm{\nabla_x g(\pb,\xb)}+1/k\le L:=\norm{\nabla_x g(\pb,\xb)}+1$ for all $x\in{\rm int}\B(\xb;R_k)$, we conclude that $g(p_k,\cdot)$ is Lipschitz continuous on ${\rm int}\B(\xb;R_k)$ with the modulus $L$ defined above, and therefore
\begin{eqnarray*}
{\rm dist}\big(g_1(p_k,\Tilde x_k);C_1\big)&=&{\rm dist}\big(g(p_k,\Tilde x_k);C\big)\le{\rm dist}\big(g(p_k,x_k);C\big)+L\norm{x_k-\Tilde x_k}\\
&\le&(L\kappa_2+1){\rm dist}\big(g(p_k,x_k);C\big).
\end{eqnarray*}
Further, it follows that ${\rm dist}(\Tilde x_k;\Gamma(p_k))\ge{\rm dist}(x_k;\Gamma(p_k))-\norm{x_k-\Tilde x_k}$, and thus we arrive at the estimates
\begin{eqnarray}\label{EqAux1}
\begin{array}{ll}
{\rm dist}\big(\Tilde x_k;\Gamma(p_k)\big)\ge{\rm dist}\big(x_k;\Gamma(p_k)\big)-\norm{x_k-\Tilde x_k}>(k+\kappa_2){\rm dist}\big(g(p_k,x_k);C\big)\\
-\norm{x_k-\Tilde x_k}\ge k\,{\rm dist}\big(g(p_k,x_k);C\big)\ge\disp\frac{k}{L\kappa_2+1}{\rm dist}\big(g(p_k,\Tilde x_k);C\big).
\end{array}
\end{eqnarray}
Now for any fixed $k\in\N$ define the positive number
\begin{eqnarray}\label{sigma}
\sigma_k:=\disp\frac{(L\kappa_2+1)^2}{k^2{\rm dist}\big(g(p_k,\Tilde x_k);C\big)}
\end{eqnarray}
and let $(\xb_k,\bar y_k)$ be an {\em optimal solution} to the problem of minimizing
\begin{eqnarray}\label{EqOptProb1}
\phi_k(x,y):=\norm{y}+\sigma_k\norm{x-\Tilde x_k}^2\;\mbox{ subject to }\;g_1(p_k,x)+y\in C_1,\;g_2(p_k,x)\in C_2.
\end{eqnarray}
Take $\Tilde y_k$ such that $g_1(p_k,\Tilde x_k)+\Tilde y_k\in C_1$ and $\norm{\Tilde y_k}={\rm dist}(g_1(p_k,\Tilde x_k);C_1)={\rm dist}(g(p_k,\Tilde x_k);C)$.
Since $(\Tilde x_k,\Tilde y_k)$ is a feasible solution to \eqref{EqOptProb1}, we get
\begin{eqnarray}\label{EqBnd_y_k1}
\norm{\oy_k}\le\phi_k(\bar x_k,\bar y_k)\le\phi_k(\Tilde x_k,\Tilde y_k)={\rm dist}\big(g(p_k,\tilde x_k);C\big).
\end{eqnarray}
It follows that $\oy_k\ne 0$ since otherwise $\xb_k\in\Gamma(p_k)$ while implying that
\begin{eqnarray*}
\frac{(L\kappa_2+1)^2}{k^2{\rm dist}(g(p_k,\Tilde x_k);C)}\big({\rm dist}(\Tilde x_k;\Gamma(p_k)\big)^2&=&\sigma_k{\rm dist}(\Tilde x_k;\Gamma(p_k))^2
\le\sigma_k\norm{\xb_k-\Tilde x_k}^2\\
&=&\phi_k(\xb_k,0)\le{\rm dist}\big(g(p_k,\Tilde x_k);C\big),
\end{eqnarray*}
which contradicts the last inequality in \eqref{EqAux1}. We have furthermore by the choice of $\sigma_k$ in \eqref{sigma} that
\begin{eqnarray*}
\sigma_k\norm{\xb_k-\Tilde x_k}^2=\disp\frac{(L\kappa_2+1)^2}{k^2{\rm dist}\big(g(p_k,\Tilde x_k);C\big)}\norm{\xb_k-\Tilde x_k}^2\le\phi_k(\xb_k,\bar y_k)\le{\rm dist}\big(g(p_k,\Tilde x_k);C\big)
\end{eqnarray*}
yielding in turn the following estimates for all $k\in\N$:
\begin{eqnarray}\label{EqBnd_xbk}
\norm{\xb_k-\Tilde x_k}
&\le&\disp\frac k{L\kappa_2+1}{\rm dist}\big(g(p_k,\Tilde x_k);C\big)
\le k\,{\rm dist}\big(g(p_k,x_k);C\big)\le k\norm{g(p_k,x_k)-g(\pb,\xb)}\qquad\\
\nonumber&<&\frac{R_k}{4k}.
\end{eqnarray}
Thus we get $\norm{\xb_k-\xb}\le\norm{\xb_k-\Tilde x_k}+\norm{\Tilde x_k-\xb}<R_k/4k+R_k/2\le\frac 34 R_k$ and
\begin{eqnarray*}
\norm{g(p_k,\xb_k)-g(\pb,\xb)}&\le&\norm{g(p_k,x_k)-g(\pb,\xb)}+\norm{g(p_k,\xb_k)-g(p_k,x_k)}\\
&\le&\frac{R_k}{4k^2\max\{\kappa_2,1\}}+L(\norm{\xb_k-\Tilde x_k}+\norm{\Tilde x_k-x_k})\\
&\le&\frac{R_k}{4k}\Big(\frac 1{k\max\{\kappa_2,1\}}+L+\frac Lk\Big)\to 0\;\mbox{ as }\;k\to\infty.
\end{eqnarray*}
Letting now $t_k:=\norm{\xb_k-\xb}+\norm{g(p_k,\xb)-\gb}$ for $k\in\N$ and passing to a subsequence if necessary allows us to claim that the sequence of $(\xb_k-\xb,g(p_k,\xb)-\gb)/t_k$ converges to some $(\bar u,\bar v)$. Then $(\bar u,\bar v)\ne(0,0)$ with $\bar v\in\ImDp$ by \eqref{Dp}. It shows furthermore that
\begin{eqnarray}\nonumber
\lefteqn{\lim_{k\to\infty}\frac{\norm{g(p_k,\xb_k)-\big(\gb+t_k(\bar v+\gxb\bar u)\big)}}{t_k}=\lim_{k\to\infty}\frac{\norm{g(p_k,\xb_k)-(g(p_k,\xb)+t_k\gxb \bar u)}}{t_k}}\\
\label{EqDirApprox}&=&\lim_{k\to\infty}\frac{\norm{g(p_k,\xb_k)-\big(g(p_k,\xb)+\gxb(\xb_k-\xb)\big)}}{t_k}\\
\nonumber&=&\lim_{k\to\infty}\frac{\Big\|\int_0^1\big[\nabla_x g\big(p_k,\xb+\xi(\xb_k-\xb)\big)-\gxb\big](\xb_k-\xb){\rm d}\xi\Big\|}{t_k}\le\lim_{k\to\infty}\frac{\norm{\xb_k-\xb}}{kt_k}=0.
\qquad\qquad\quad\end{eqnarray}

Our next step is to prove that the solution $\oy_k$ to the optimization problem \eqref{EqOptProb1} satisfies
\begin{eqnarray}\label{EqAuxTangDir}
\lim_{k\to\infty}\frac{\norm{\oy_k}}{t_k}=0.
\end{eqnarray}
Assume on the contrary that there is $\epsilon>0$ such that after passing to some subsequence we have
\begin{equation}\label{EqContraDic1}
\norm{\oy_k}\ge\epsilon t_k,\quad k\in\N.
\end{equation}
For every $k$ sufficiently large find $j(k)\le k$ with $t_k\norm{u}\le R_{j(k)}$ such that $\lim_{k\to\infty}j(k)=\infty$ and therefore
\begin{eqnarray*}
\norm{g(p_k,\xb+t_k u)-\big(g(p_k,\xb)+t_k\gxb u\big)}&=&\Big\|\disp\int_0^1 t_k\Big(\nabla_x g(p_k,\xb+\xi t_k u)-\gxb\Big)u{\rm d}\xi\Big\|\\
&\le&\frac{t_k\norm{u}}{j(k)},
\end{eqnarray*}
which implies by \eqref{EqExistenceTangDir} that $\liminf_{k\to\infty}{\rm dist}(g(p_k,\xb+t_k u);C)/t_k=0$. After passing to a subsequence we can assume that ${\rm dist}(g(p_k,\xb+t_k u);C)< t_k/k$ for all $k\in\N$ and then get the conditions $\norm{\Hat x_k-(\xb+t_k u)}\le\kappa_2t_k/k$ and $g_2(p_k,\Hat x_k)\in C_2$ for some $\Hat x_k$. This tells us that
$$
{\rm dist}\big(g(p_k,\Hat x_k);C\big)={\rm dist}\big(g_1(p_k,\Hat x_k);C_1\big)\le(1+L\kappa_2)t_k/k,\quad k\in\N.
$$
Picking $\Hat y_k$ with $g_1(p_k,\hat x_k)+\Hat y_k\in C_1$ and ${\rm dist}(g_1(p_k,\Hat x_k);C_1)=\norm{\Hat y_k}$, get by $\phi_k(\xb_k,\oy_k)\le\phi_k(\Hat x_k,\Hat y_k)$ that
\begin{eqnarray}\label{EqBnd_y_k2}
\begin{array}{ll}
\norm{\oy_k}&\le\norm{\Hat y_k}+\sigma_k\big(2\skalp{\xb_k-\Tilde x_k,\Hat x_k-\xb_k}+\norm{\Hat x_k-\xb_k}^2\big)\\
&\le(1+L\kappa_2)t_k/k+2\sigma_k\norm{\xb_k-\Tilde x_k}\cdot\norm{\Hat x_k-\xb_k}+\sigma_k\norm{\Hat x_k-\xb_k}^2
\end{array}
\end{eqnarray}
and then deduce from \eqref{EqBnd_xbk} the relationships
\begin{eqnarray}\label{EqBndSigma}
\disp\sigma_k\norm{\xb_k-\Tilde x_k}\le\frac{(L\kappa_2+1)^2}{k^2{\rm dist}\big(g(p_k,\Tilde x_k);C\big)}\frac{k\,{\rm dist}\big(g(p_k,\Tilde x_k);C\big)}{L\kappa_2+1}=\frac{L\kappa_2+1}{k}\to 0\;\mbox{ as }\;k\to\infty.\qquad
\end{eqnarray}
Using $\norm{\Hat x_k-\xb_k}\le\norm{\Hat x_k-\xb}+\norm{\xb_k-\xb}\le t_k(\norm{u}+\kappa_2/k+1)$ and combining \eqref{EqBnd_y_k1} and \eqref{EqBnd_y_k2} yield
\begin{eqnarray*}
\norm{\oy_k}&\le&\min\big\{(1+L\kappa_2)t_k/k+2\sigma_k\norm{\xb_k-\Tilde x_k}\cdot\norm{\Hat x_k-\xb_k}+\sigma_k\norm{\Hat x_k-\xb_k}^2,{\rm dist}\big(g(p_k,\Tilde x_k);C\big)\big\}\\
&\le&\frac{t_k}k(L\kappa_2+1)\Big(2\norm{u}+2\frac{\kappa_2}{k}+3\Big)+\min\big\{\sigma_k\norm{\Hat x_k-\xb_k}^2,{\rm dist}\big(g(p_k,\Tilde x_k);C\big)\big\}\\
&\le&\frac{t_k}k(L\kappa_2+1)\Big(2\norm{u}+2\frac{\kappa_2}{k}+3\Big)+\sqrt{\sigma_k{\rm dist}\big(g(p_k,\Tilde x_k);C\big)}\norm{\Hat x_k-\xb_k}\\
&\le&\frac{t_k}k(L\kappa_2+1)\Big(3\norm{u}+3\frac{\kappa_2}{k}+4\Big),
\end{eqnarray*}
which contradicts \eqref{EqContraDic1} and thus justifies \eqref{EqAuxTangDir}. Since $g(p_k,\xb_k)+(\oy_k,0)\in C$, the results obtained in \eqref{EqDirApprox}, \eqref{EqAuxTangDir} and definition \eqref{tan} of the contingent cone allow us to conclude that
\begin{eqnarray}\label{ctan}
\bar v+\gxb\bar u\in T_C\big(\gb\big).
\end{eqnarray}

Let us next show that the constraint mapping $G_k(x,y):=(g_1(p_k,x)+y-C_1,g_2(p_k,x)-C_2)$ of program \eqref{EqOptProb1} is {\em metrically subregular} at $((\xb_k,y_k),0)$. Indeed, pick $(x,y)\in{\rm int}\B(\xb_k;\frac{R_k}{4\max\{\kappa_2L,1\}})\times\R^{l_1}$ and find $\xi\in\Gamma_2(p_k)$ with $\norm{\xi-x}\le\kappa_2{\rm dist}(g_2(p_k,x);C_2)$. Since $g_2(p_k,\xb_k)\in C_2$, it gives us
\begin{eqnarray*}
\norm{\xi-x}\le\kappa_2\norm{g_2(p_k,x)-g_2(p_k,\xb_k)}\le\kappa_2L\norm{x-\xb_k}<\frac{R_k}4,
\end{eqnarray*}
and consequently we have $\xi\in{\rm int}\B(\xb;R_k)$ and $\norm{g_1(p_k,\xi)-g_1(p_k,x)}\le L\norm{\xi-x}$. Consider now a solution $\bar\vartheta\in\R^{l_1}$ to the following optimization problem:
\begin{eqnarray*}
\mbox{minimize }\;\norm{\vartheta-y}\;\mbox{ subject to }\;g_1(p_k,\xi)+\vartheta\in C_1.
\end{eqnarray*}
Then $0\in G_k(\xi,\bar\vartheta)$ and, with $\vartheta'\in C_1-(g_1(p_k,x)+y)$ such that $\norm{\vartheta'}={\rm dist}(g_1(p_k,x)+y;C_1)$, we get
\begin{eqnarray*}
\norm{\bar\vartheta-y}&\le&\norm{g_1(p_k,x)+y+\vartheta'-g_1(p_k,\xi)-y)}\le L\norm{\xi-x}+\norm{\vartheta'}.
\end{eqnarray*}
Thus it verifies the metric subregularity of $G_k$ at $((\xb_k,y_k),0)$ by
\begin{eqnarray*}
\norm{\xi-x}+\norm{\bar\vartheta-y}&\le&\rm dist\big(g_1(p_k,x)+y;C_1\big)+(L+1)\norm{\xi-x}\\
&\le&\big((L+1)\kappa_2+1\big)\big({\rm dist}\big(g_1(p_k,x)+y;C_1\big)+{\rm dist}\big(g_2(p_k,x);C_2\big)\big)\\
&\le&\sqrt{2}\big((L+1)\kappa_2+1\big){\rm dist}\big(0;G_k(x,y)\big).
\end{eqnarray*}

Since metric subregularity is a {\em constraint qualification} (MSCQ) for NLPs, we apply to problem \eqref{EqOptProb1} the well-recognized necessary optimality conditions via limiting normals at $(\xb_k,\oy_k)$ (cf.\ \cite{Mor06a,RoWe98}): there exist multipliers $\lambda^1_k\in N_{C_1}(g_1(p_k,\xb_k)+\oy_k)$ and $\lambda_k^2\in  N_{C_2}(g_2(p_k,\xb_k))$ such that
\begin{eqnarray}\label{EqAuxOptCond}
2\sigma_k(\xb_k-\Tilde x_k)+\nabla_x g_1(p_k,\xb_k)^*\lambda^1_k+\nabla_x g_2(p_k,\xb_k)^*\lambda^2_k=0\;\mbox{ and }\;\disp\frac{\oy_k}{\norm{\oy_k}}+\lambda^1_k=0.
\end{eqnarray}
Remembering by Theorem~\ref{ThEquMSP_BMP} that Robinson stability of \eqref{split} implies the partial BMP with respect to $x$ at the corresponding points, for each $k$ sufficiently large we can choose the multiplier $\lambda^2_k$ satisfying
\begin{eqnarray*}
\norm{\lambda^2_k}\le(\kappa_2+1)\big(\sigma_k\norm{\xb_k-\Tilde x_k}+\norm{\nabla_x g_1(p_k,\xb_k)^*\lambda^1_k}\big).
\end{eqnarray*}
It follows from \eqref{EqBndSigma} and from $\norm{\lambda^1_k}=1$ due to \eqref{EqAuxOptCond} that the sequence of $(\lambda^1_k,\lambda^2_k)$ is bounded and thus its subsequence converges to some $\lambda=(\lambda^1,\lambda^2)$ with $\norm{\lambda^1}=1$. By taking \eqref{ctan} into account, we conclude that $\lambda\in N_C(g(\pb,\xb);\bar v+\gxb \bar u)=\prod_{i=1}^2 N_{C_i}(g_i(\pb,\xb);\bar v_i+\nabla_xg_i(\pb,\xb)\bar u)$. Using finally \eqref{EqAuxOptCond} together with \eqref{EqBndSigma} tells us that
\begin{eqnarray*}
\nabla_x g(\pb,\xb)^*\lambda=\nabla_x g_1(\pb,\xb)^*\lambda^1+\nabla_x g_2(\pb,\xb)^*\lambda^2=0,
\end{eqnarray*}
which contradicts the assumed condition \eqref{EqPartFOReg} and thus completes the proof of the theorem.\end{proof}

Next we present several consequences of Theorem~\ref{ThPartFOReg} referring the reader to Sections~4 and 5 for further applications. Let us first formulate a version of the theorem without splitting system \eqref{EqConstrSystem} into two subsystems, i.e., with $l_2=0$ in \eqref{split}.

\begin{Corollary}[verification of Robinson stability without splitting]\label{CorPartFOREg} Let $l_2=0$ in the framework of Theorem~{\rm\ref{ThPartFOReg}}, i.e., in addition to condition \eqref{EqExistenceTangDir} the implication
\begin{eqnarray}\label{EqPartFORegCor}
\big[\lambda\in N_C\big(g(\pb,\xb);v+\nabla_xg(\pb,\xb)u\big),\;\nabla_x g(\pb,\xb)^*\lambda=0\big]\Longrightarrow\lambda=0
\end{eqnarray}
holds for every $(0,0)\not=(v,u)\in\ImDp\times\R^n$ with $v+\gxb u\in T_C(\gb)$. Then system \eqref{EqConstrSystem} enjoys the Robinson stability property at
$(\pb,\xb)$.
\end{Corollary}

It is easy to check that all the assumptions of Corollary~\ref{CorPartFOREg} are satisfied under the metric regularity of the underlying mapping $x\mapsto g(\op,x)-C$ around the reference point $(\ox,0)$.

\begin{Corollary}[Robinson stability from metric regularity]\label{rob-mr} If the mapping $g(\pb,\cdot)-C$ is metrically regular around $(\xb,0)$, then Robinson stability holds for system \eqref{EqConstrSystem} at $(\pb,\xb)$.
\end{Corollary}
\begin{proof} Recall that $g(\pb,\cdot)-C$ is metrically regular around $(\xb,0)$ if and only if the mapping $u\mapsto\gb+\gxb u-C$ is metrically regular around $(0,0)$, cf.\ \cite[Corollary~3F.5]{DoRo14}. Then for any $v\in\R^l$ and any sequence $t_k\dn 0$ we can find $u_k$ with $\gb+t_k v+\gxb u_k\in C$ and $\norm{u_k}\le\kappa\,{\rm dist}(\gb+t_k v\;C)\le t_k\norm{v}$ whenever $k\in\N$ is sufficiently large. Hence the sequence $\{u_k/t_k\}$ is bounded and its subsequence $\{u_{k_i}/t_{k_i}\}$ converges to some $u\in\R^n$, and so
\begin{eqnarray*}
\liminf_{k\to\infty}\frac{{\rm dist}\big(\gb+t_k(v+\gxb u);C\big)}{t_k}&\le&\lim_{i\to\infty}\frac{{\rm dist}\big(\gb+t_{k_i}\big(v+\gxb \frac{u_{k_i}}{t_{k_i}}\big);C\big)}{t_{k_i}}=0.
\end{eqnarray*}
Thus assumption \eqref{EqExistenceTangDir} is satisfied in this setting. The validity of implication \eqref{EqPartFORegCor} follows immediately from the metric regularity characterization in \eqref{mr-f}.\end{proof}

To conclude this section, we discuss some important settings where the major assumptions of Theorem~\ref{ThPartFOReg} are satisfied without imposing metric regularity of the mapping $g(\op,\cdot)-C$.

\begin{Remark}[assumption verification]\label{RemExistenceTangDir}{\rm {\bf (i)} Let $C$ be a union of finitely many convex polyhedra $C_1,\ldots,C_m$, and let the triple $(z,u,v)$ fulfill the condition $v+\nabla\gb(z,u)\in T_C(\gb)$. Then there is $\bar t>0$ with
\begin{eqnarray*}
\gb+t\big(v+\nabla\gb(q,u)\big)\in C\;\mbox{ for all }\;t\in[0,\bar t];
\end{eqnarray*}
in particular, condition \eqref{EqExistenceTangDir} is satisfied for every sequence $t_k\dn 0$. Indeed, in this case any tangent direction $w\in T_C(\oy)$ belongs to the contingent cone of one of the sets $C_i$, and hence there exists $\bar t>0$ for which $\oy+tw\in C_i\subset C$ whenever $t\in[0,\bar t]$.

{\bf (ii)} If $C$ is convex, then the inclusion $(\lambda^1,\lambda^2)\in\prod_{i=1}^2N_{C_i}(g_i(\pb,\xb);v_i+\nabla_xg_i(\pb,\xb)u)$ appearing in \eqref{EqPartFOReg} is equivalent to the conditions:
\begin{eqnarray*}
v_i+\nabla_xg_i(\pb,\xb)u\in T_C\big(g_i(\pb,\xb)\big),\;\lambda_i\in N_C\big(g_i(\pb,\xb)\big),\;\skalp{\lambda_i,v_i+\nabla_xg_i(\pb,\xb)u}=0,\;i=1,2.
\end{eqnarray*}
This follows directly from \cite[Lemma~2.1]{Gfr14b}.}
\end{Remark}

\section{Second-Order Conditions for Robinson Stability and Subregularity}

This section is devoted to deriving verifiable {\em second-order} conditions for Robinson stability of PCS, which has never been done in the literature. Our results below take into account the curvatures of the constraint set $C$ and the parameter set $P$.

Given a closed subset $\Omega\subset\R^s$, a point $\bar z\in\Omega$, a direction $v\in T_\Omega(\bar z)$ and a multiplier $\lambda\in\R^s$, we introduce the following (directional) {\em upper curvature} and {\em lower curvature} of $\O$, respectively:
\begin{eqnarray}\label{EqCurvP}
&&\ochi_\Omega(\lambda,\bar z;v):=\lim_{\epsilon\dn 0}\sup\Big\{\frac{\skalp{\lambda,v'-v}}{\tau}\Big|\;0<\tau<\epsilon,\;\norm{v'-v}<\epsilon,\;\bar z+\tau v'\in\Omega\Big\},\\
\label{EqCurvC}
&&\uchi_\Omega(\lambda,\bar z;v):=\lim_{\epsilon\dn 0}\inf\Big\{\frac{\skalp{\lambda,v'-v}}{\tau}\Big|0<\tau<\epsilon,\norm{v'-v}<\epsilon,{\rm dist}\big(\lambda;N_\Omega(\bar z+\tau v')\big)<\epsilon\Big\}.\qquad
\end{eqnarray}
Observe that both $\ochi_\Omega(\lambda,\bar z;v)$ and $\uchi_\Omega(\lambda,\bar z;v)$ can have values $\pm\infty$ and that $\uchi_\Omega(\lambda,\bar z;v)=\infty$ if $\lambda\not\in N_\Omega(\bar z;v)$. Otherwise we clearly have the relationship $\ochi_\Omega(\lambda,\bar z;v)\geq\uchi_\Omega(\lambda,\bar z;v)$. Note also that some related while different curvature quantities were used in the literature for deriving second-order  optimality conditions in nonconvex problems of constrained optimization, see, e.g., \cite{ap06, BonSh00, Pen98}.

Recall \cite{BonSh00} that, given a closed set $\Omega\subset\R^s$, the {\em outer second-order tangent set} to $\O$ at $\bar z\in\Omega$ in direction $v\in T_\Omega(\bar z)$ is defined by
\begin{eqnarray*}
T_\Omega^2(\bar z;v):=\limsup_{\tau\dn 0}\frac{\Omega-\zb-\tau v}{\frac 12\tau^2}.
\end{eqnarray*}

\begin{Proposition}[upper curvature via second-order tangent set]\label{LemUpSecOrd} We have the relationship
\begin{eqnarray*}
\ochi_\Omega(\lambda,\bar z;v)\ge\frac12\sup\Big\{\la\lambda,w\ra\Big|\;w\in T_\Omega^2(\bar z;v)\Big\}.
\end{eqnarray*}
\end{Proposition}
\begin{proof} Letting $\eta:=\sup\{\skalp{\lambda,w}|\;w\in T_\Omega^2(\bar z;v)\}$, observe that $\eta=-\infty$ if $T_\Omega^2(\bar z;v)=\emp$, and thus the statement is trivial in this case. When $T_\Omega^2(\bar z;v)\ne\emp$, for an arbitrarily fixed $\delta>0$ consider $w\in T_\Omega^2(\bar z;v)$ satisfying $\skalp{\lambda,w}\ge\eta-\delta$ if $\eta<\infty$ and $\skalp{\lambda,w}\ge 1/\delta$ if $\eta=\infty$. Then we can find sequences $\tau_k\dn 0$ and $z_k\setto{\Omega}\bar z$ with $2(v_k-v)/\tau_k\to w$, where $v_k:=(z_k-\bar z)/\tau_k$. Since $v_k\to v$, for any $\epsilon>0$ there is $k\in\N$ such that $\tau_k<\epsilon$, $\norm{v_k-v}<\epsilon$, and $2\skalp{\lambda,(v_k-v)/\tau_k}\ge\skalp{\lambda,w}-\epsilon$. By $z_k=\zb+t_kv_k\in\Omega$ we conclude that
\begin{eqnarray*}
\disp\sup\Big\{\frac{\skalp{\lambda,v'-v}}{\tau}\Big|\;0<\tau<\epsilon,\;\norm{v'-v}\le\epsilon,\;\bar z+\tau v'\in\Omega\}\ge\skalp{\lambda,(v_k-v)/\tau_k}\ge\frac 12(\skalp{\lambda,w}-\epsilon),
\end{eqnarray*}
which therefore completes the proof of the proposition. \end{proof}

The next important result provides explicit evaluations for the upper and lower curvatures of sets under a {\em certain subamenability}. In fact, the first statement of the following theorem holds for {\em strongly subamenable} sets from Definition~\ref{subamen}, while the second statement covers {\em fully subamenable} sets if the image set $Q$ in the representation below is just one convex polyhedron (instead of their finite unions).

\begin{Theorem}[upper and lower curvatures of set under subamenability]\label{PropSetSecOrd} Let $\Omega:=\{z\in\R^s|\;q(z)\in Q\}$, where $q:\R^s\to\R^p$ is twice differentiable at $\bar z\in\Omega$, $Q\subset\R^p$ is a closed set, and where the mapping $q(\cdot)-Q$ is metrically subregular at $(\bar z,0)$. Given $v\in T_{\Omega}(\bar z)$, we have the assertions:

{\bf (i)} If $Q$ is convex, $\lambda\in N_\Omega(\bar z)$, and $\skalp{\lambda,v}=0$, then
\begin{eqnarray}\label{EqUpSecOrd}
\ochi_\Omega(\lambda,\bar z;v)\le-\disp\frac 12\sup\Big\{\skalp{\nabla^2\skalp{\mu,q}(\bar z)v,v}\Big|\;\mu\in N_Q\big(q(\bar z)\big),\;\lambda=\nabla q(\bar z)^*\mu\Big\}.
\end{eqnarray}

{\bf (ii)} If $\lambda\in N_\Omega(\bar z;v)$ and $Q$ is the union of finitely many convex polyhedra, then there is a vector $\mu\in N_Q\big(q(\bar z);\nabla q(\bar z)v\big)$ such that $\lambda=\nabla q(\bar z)^*\mu$ and that
\begin{eqnarray}\label{EqLowSecOrd}
\uchi_\Omega(\lambda,\bar z;v)=-\disp\frac 12\skalp{\nabla^2\skalp{\mu,q}(\bar z)v,v}.
\end{eqnarray}
\end{Theorem}
\begin{proof}
To verify (i), consider sequences $\tau_k\dn 0$ and $v_k\to v$ with $\bar z+\tau_k v_k\in\Omega$, $k\in\N$, such that $\skalp{\lambda,v_k-v}/\tau_k\to\ochi_\Omega(\lambda,\bar z;v)$ and take $\mu\in N_Q(q(\bar z))$ with $\nabla q(\bar z)^*\mu=\lambda$. Then $q(\bar z+\tau_k v_k)=q(\bar z)+\tau_k\nabla q(\bar z)v_k+\frac 12\tau_k^2\skalp{\nabla^2 q(\bar z)v,v}+\oo(\tau_k^2)$ and $\skalp{\mu,q(\bar z+\tau_kv_k)-q(\bar z)}\le 0$ by the convexity of $Q$ thus yielding
\begin{eqnarray*}
\disp\frac{\skalp{\lambda, v_k-v}}{\tau_k}=\frac{\skalp{\mu,\nabla q(\xb)v_k}}{\tau_k}\le-\frac 12\skalp{\nabla^2\skalp{\mu,q}(\bar z)v,v}+\frac{\oo(\tau_k^2)}{\tau_k^2}.
\end{eqnarray*}
This readily justifies \eqref{EqUpSecOrd} by definition \eqref{EqCurvC} of the upper curvature.

To proceed with the verification of (ii), take sequences $\tau_k\dn 0$, $v_k\to v$ and $\lambda_k\to\lambda$ as $k\to\infty$ with $\lambda_k\in N_\Omega(\bar z+\tau_k v_k)$ for all $k\in\N$ such that $\skalp{\lambda, v_k-v}/\tau_k\to\uchi_\Omega(\lambda,\bar z;v)$. It follows from Lemma~\ref{CorBMP} by the assumed metric subregularity that there is $\kappa>0$ such that for each $k$ sufficiently large we can find $\mu_k\in N_Q(q(\bar z+\tau_k v_k))\cap\kappa\norm{\lambda_k}\B_{\R^{p}}$ with $\lambda_k=\nabla q(\bar z+\tau_kv_k)^*\mu_k$. Hence the sequence of $\mu_k$ is bounded and its subsequence converges to some $\mu\in N_Q(q(\bar z);\nabla q(\bar z)v)$ satisfying $\lambda=\nabla q(\bar z)^*\mu$. Then \cite[Lemma~3.4]{Gfr13b} allows us to find $d_k\in Q$ with $\norm{d_k-q(\bar z+\tau_kv_k)}\le\tau_k^3$ such that $\mu\in\Hat N_Q(d_k)$.

Remembering that $Q$ is the union of the convex polyhedra $Q_1,\ldots,Q_m$ having the representations $Q_i=\{d\in\R^p|\;\skalp{a_{ij},d}\le b_{ij},\ j=1,\ldots,m_i\}$ for $i=1,\ldots,m$, we get
$$
\Hat N_Q(d_k)=\bigcap_{i\in J_k}\Hat N_{Q_i}(d_k)\;\mbox{ with }\;J_k:=\big\{i\in\{1,\ldots,m\}\big|\;d_k\in Q_i\big\}.
$$
Since $J_k\ne\emp$ for each $k\in\N$, there is a subsequence $\{k\}$ and some index $\Hat i$ such that $\Hat i\in J_k$ along this subsequence. By passing to a subsequence again (no relabeling), we can suppose that the index sets ${\cal J}_{\Hat i}(d_k):=\{j|\;\skalp{a_{{\Hat i}j},d}= b_{\Hat ij}\}$ reduces to a constant set ${\cal J}$ for all $k\in\N$. Employing now the Generalized Farkas Lemma from \cite[Proposition~2.201]{BonSh00} gives  us a constant $\beta\ge 0$ such that for every $k$ there are numbers $\nu_{kj}\ge 0$ as $j\in{\cal J}$ for which
$$
\mu_k=\disp\sum_{j\in{\cal J}}a_{{\Hat i}j}\nu_{kj}\;\mbox{ and }\;\sum_{j\in{\cal J}}\nu_{kj}\le\beta\norm{\mu_k}.
$$
Thus the sequences $\{\nu_{kj}\}$ for all $j\in{\cal J}$ are bounded, and the passage to a subsequence tells us that for every $j\in{\cal J}$ the sequence of $\nu_{kj}$ converges to some $\nu_j$ as $k\to\infty$. Then $\nu_j\ge 0$, $\mu=\sum_{j\in{\cal J}}a_{{\Hat i}j}\nu_j$, and $\mu\in N_{Q_{\Hat i}}(d_k)$ for all $k$. Hence $\skalp{\mu,q(\bar z)-d_k}\le 0$ and, since we also have $\mu\in N_{Q_{\Hat i}}(q(\bar z))$, it follows that $\skalp{\mu,q(\bar z)-d_k}=0$. Furthermore, taking into account the representations
$$
d_k-q(\bar z)=q(\bar z+\tau_k v_k)-q(\bar z)+\oo(\tau_k^2)=\tau_k\nabla q(\bar z)v_k+\disp\frac 12\tau_k^2\skalp{\nabla^2 q(\bar z)v,v}+\oo(\tau_k^2)
$$
allows us finally to arrive at the relationships
\begin{eqnarray*}
\skalp{\lambda,v}=\skalp{\mu,\nabla q(\bar z)v}=\disp\lim_{k\to\infty}\frac{\skalp{\mu,d_k-q(\bar z)}}{\tau_k}=0,
\end{eqnarray*}
\begin{eqnarray*}
\disp\frac{\skalp{\lambda, v_k-v}}{\tau_k}=\frac{\skalp{\mu,\nabla q(\zb)v_k}}{\tau_k}=-\frac 12\skalp{\nabla^2\skalp{\mu,q}(\bar z)v,v}+\frac{\oo(\tau_k^2)}{\tau_k^2},
\end{eqnarray*}
which complete the proof of the theorem by recalling definition \eqref{EqCurvC} of the lower curvature.\end{proof}

The next theorem is the major result of this section. For simplicity we restrict ourselves to the case where the parameter space $P$ is finite-dimensional.

\begin{Theorem}[second-order verification of Robinson stability]\label{ThPartSOReg} Consider the splitting system \eqref{split}, where $P\subset\R^m$ in our standing assumptions. Suppose also that for every $w\in T_P(\bar p)$ and for every sequence ${\tau_k}\dn 0$ there exists $u\in\R^n$ with
\begin{eqnarray}\label{EqTangDir}
\disp\liminf_{k\to\infty}{\rm dist}\big(\gb+\tau_k\nabla\gb(w,u);C\big)/\tau_k=0,
\end{eqnarray}
that Robinson stability holds at $(\pb,\xb)$ for the system $g_2(p,x)\in C_2$ in \eqref{split}, and that for every triple $(w,u,\lambda)\in\R^m\times\R^n\times\R^l$ satisfying the conditions
\begin{eqnarray}
\label{EqSOCond1}&&(0,0)\not=(w,u),\;w\in T_P(\pb),\;\nabla\gb(w,u)\in T_C\big(\gb\big),\\
\label{EqSOCond2}&&\lambda\in N_C\big(\gb;\nabla\gb(w,u)\big),\;\nabla_x g(\pb,\xb)^*\lambda=0,\\
\label{EqSOCond3}&&\frac 12\skalp{\nabla^2\skalp{\lambda,g}(\pb,\xb)(w,u),(w,u)}+\ochi_P\big(\gpb^*\lambda,\pb;w\big)\ge\uchi_C\big(\lambda,\gb;\nabla\gb(w,u)\big)\qquad
\end{eqnarray}
we have $\lambda^1=0$, where $\lambda=(\lambda^1,\lambda^2)\in\R^{l_1}\times\R^{l_2}$. Then the Robinson stability property at $(\op,\ox)$ also holds for the whole system $g(p,x)\in C$ in \eqref{split}.
\end{Theorem}
\begin{proof}
Assuming on the contrary that Robinson stability fails at $(\op,\ox)$ for the system $g(p,x)\in C$ in \eqref{split} and taking $\zeta(p):=\norm{p-\pb}$, we proceed as in the proof of Theorem~\ref{ThPartFOReg} and find sequences $\xb_k$, $y_k$, and $\lambda_k:=(\lambda^1_k,\lambda^2_k)$ such that the limit $\lambda=(\lambda^1,\lambda^2)=\lim_{k\to\infty} \lambda_k$ satisfies the relationships
\begin{eqnarray*}
\lambda_k\in N_C(c_k),\;\lambda^1_k=-\disp\frac{y_k}{\norm{y_k}},\;\gxb^*\lambda=0,\;\norm{\lambda^1}=1,
\end{eqnarray*}
where $c_k:=g(p_k,\xb_k)+(y_k,0)\in C$. For each $k\in\N$ choosing $R_k$ as above, define $0<\tau_k:=\norm{p_k-\pb}+\norm{\xb_k-\xb}\le\frac 1k+\frac 34 R_k<\frac 2k$ and suppose by passing to a subsequence that the sequence $\{(p_k-\pb,\xb_k-\xb)/\tau_k\}$ converges to some $\Tilde z:=(\Tilde w,\Tilde u)\not=(0,0)$. Let us show that the triple $(\Tilde w,\Tilde u,\lambda)$ satisfies conditions \eqref{EqSOCond1}--\eqref{EqSOCond3}, which contradicts the assumption of the theorem due to $\lambda^1\ne 0$.

We obviously have $\Tilde w\in T_P(\pb)$ and $\lim_{k\to\infty}(g(p_k,\xb)-\gb)/\tau_k=\gpb\Tilde w=:\Tilde v$ together with
\begin{eqnarray*}
\disp\lim_{k\to\infty}\frac{g(p_k,\xb_k)-\gb}{\tau_k}=\nabla\gb\Tilde z=\Tilde v+\gxb\Tilde u.
\end{eqnarray*}
Now we proceed as in the proof of Theorem~\ref{ThPartFOReg} to verify \eqref{EqAuxTangDir}. Using the same arguments as in Theorem~\ref{ThPartFOReg} with replacing  $\bar v$, $\bar u$, and $t_k$ by $\Tilde v,\Tilde u$, and $\tau_k$, respectively, implies that $y_k/\tau_k\to 0$. Further, by setting $s_k:=(c_k-\gb)/\tau_k$ we obtain
\begin{eqnarray*}
\disp\lim_{k\to\infty}s_k=\lim_{k\to\infty}\left(\frac{g(p_k,\xb_k)-\gb}{\tau_k}+\frac{(y_k,0)}{\tau_k}\right)=\nabla \gb\Tilde z,
\end{eqnarray*}
which gives us $\nabla\gb\Tilde z\in T_C(\gb)$ and $\lambda\in N_C(\gb;\nabla\gb\Tilde z)$. Hence \eqref{EqSOCond1} and \eqref{EqSOCond2} are fulfilled, and it remains to justify \eqref{EqSOCond3}. By passing to a subsequence, suppose the validity of
\begin{eqnarray}\label{EqAuxSOBnd1}
\disp\frac{\skalp{\lambda,c_k-\gb-\tau_k\nabla\gb\Tilde z}}{\tau_k^2}=\frac{\skalp{\lambda,s_k-\nabla\gb\Tilde z}}{\tau_k}\ge\uchi_C\big(\lambda,\gb;\nabla\gb\Tilde z\big)-
\disp\frac 1k
\end{eqnarray}
for all $k\in\N$. By $-1=\skalp{\lambda^1_k,y_k/\norm{y_k}}$ we also get
\begin{eqnarray*}
-\disp\frac 12\ge\skalp{\lambda^1,y_k/\norm{y_k}}=\frac{\skalp{\lambda,c_k-g(p_k,\xb_k)}}{\norm{y_k}}\;\mbox{ and }\;\skalp{\lambda,c_k-g(p_k,\xb_k)}\le 0
\end{eqnarray*}
when $k$ is sufficiently large. Hence \eqref{EqAuxSOBnd1} yields the estimate
\begin{eqnarray*}
\disp\frac{\skalp{\lambda,g(p_k,\xb_k)-\gb-\tau_k\nabla\gb\Tilde z}}{\tau_k^2}\ge\uchi_C\big(\lambda,\gb;\nabla\gb\Tilde z\Big)-\frac 1k,
\end{eqnarray*}
which implies, by passing to a subsequence if necessary, that
\begin{eqnarray*}
\disp\frac 12\skalp{\nabla^2\skalp{\lambda,g}(\pb,\xb)\Tilde z,\Tilde z}+\frac{\skalp{\lambda,\nabla\gb(\xb_k-\xb-\tau_k\Tilde u,p_k-\pb-\tau_k\Tilde w)}}{\tau_k^2}\ge\disp\uchi_C\big(\lambda,\gb;\nabla\gb\Tilde z\big)-\frac 2k
\end{eqnarray*}
for all $k$. Setting $w_k:=(p_k-\pb)/\tau_k$ as $k\in\N$, we have
\begin{eqnarray*}
\frac{\skalp{\lambda,\gpb(p_k-\pb-\tau_k\Tilde w)}}{\tau_k^2}&=&\frac{\skalp{\lambda,\gpb(w_k-\Tilde w)}}{\tau_k}=\frac{\skalp{\gpb^*\lambda,w_k-\Tilde w)}}{\tau_k}\\
&\le&\ochi_P\big(\gpb^*\lambda,\gb;\Tilde w)+\frac 1k
\end{eqnarray*}
when $k$ is sufficiently large. Taking into account that $\skalp{\lambda,\nabla\gb(\xb_k-\xb-\tau_k\Tilde u)}=\skalp{\gxb^*\lambda,\xb_k-\xb-\tau_k\Tilde u}=0$, this gives us the estimate
\begin{eqnarray*}
\disp\frac 12\skalp{\nabla^2\skalp{\lambda,g}(\pb,\xb)\Tilde z,\Tilde z}+\ochi_P\big(\gpb^\ast\lambda,\gb;\Tilde w\big)\ge\uchi_C\big(\lambda,\gb;\nabla\gb\Tilde z\big)-
\disp\frac 3k.
\end{eqnarray*}
It shows by passing to the limit that the triple $(\Tilde w,\Tilde u,\lambda)$ satisfies \eqref{EqSOCond3}. This contradicts the assumptions of the theorem and thus completes the proof.\end{proof}

Let us now present a consequence of Theorem~\ref{ThPartSOReg} for an important special case of PCS \eqref{EqConstrSystem}, where $C$ is the union of finitely many convex polyhedra, and where the parameter space
\begin{eqnarray}\label{EqPInequ}
P:=\big\{p\in\R^m\big|\;h_i(p)\le 0\;\mbox{ for }\;i=1,\ldots,l_P\}
\end{eqnarray}
is described by smooth functions $h=(h_1,\ldots,h_{l_P})\colon\R^m\to\R^{l_P}$. Suppose for simplicity that $h(\pb)=0$.

\begin{Corollary}[second-order conditions for Robinson stability of PCS defined by unions of convex polyhedra]\label{CorSecOrdReg} Consider PCS \eqref{EqConstrSystem}, where $P\subset\R^m$ is defined by \eqref{EqPInequ}, and where the mappings $g:\R^m\times\R^n\to\R^l$ and $h:\R^m\to\R^{l_P}$ are  twice differentiable at $(\pb,\xb)\in\gph\Gamma$ and $\pb\in P$, respectively. Assume that the set $C$ is the union of finitely many convex polyhedra and that the mapping $h(\cdot)-\R^{l_P}_-$ is metrically subregular at $(\pb,0)$. Suppose also that for every $w\in\R^m$ with $\nabla h(\pb)w\le 0$ there is $u\in\R^n$ with $\nabla\gb(w,u)\in T_C(\gb)$ and that for every triple $(w,u,\lambda)$ satisfying
\begin{eqnarray}
\label{EqSOCond1a}&&(0,0)\not=(w,u),\;\nabla h(\pb)w\le 0,\;\nabla\gb(w,u)\in T_C\big(\gb\big),\\
\label{EqSOCond2a}&&0\not=\lambda\in N_C\big(\gb;\nabla\gb(w,u)\big),\;\nabla_x g(\pb,\xb)^*\lambda=0
\end{eqnarray}
there exists $\mu\in\R^{l_P}_+$ such that $\gpb^*\lambda=\nabla h(\pb)^*\mu$ and
\begin{eqnarray}\label{EqSOCond3a}
\skalp{\nabla^2\skalp{\lambda,g}(\pb,\xb)(w,u),(w,u)}-\skalp{\nabla^2\skalp{\mu,h}(\pb)w,w}<0.
\end{eqnarray}
Then the Robinson stability property holds for system \eqref{EqConstrSystem} at the point $(\pb,\xb)$ .
\end{Corollary}
\begin{proof} To verify this result, we apply Theorem~\ref{ThPartSOReg} with $l_2=0$ and $l_1=l$. Since $T_P(\pb)\subset\{w|\;\nabla h(\pb)w\le 0\}$, the imposed assumptions imply that for every $w\in T_P(\pb)$ there is $u\in\R^n$ with $\nabla\gb(w,u)\in T_C(\gb)$ and that condition \eqref{EqTangDir} holds because $C$ is the union of finitely many convex polyhedra; see Remark~\ref{RemExistenceTangDir}(i). In order to apply Theorem~\ref{ThPartSOReg}, it now suffices to show that there is no triple $(w,u,\lambda)$ fulfilling \eqref{EqSOCond1}--\eqref{EqSOCond3} with $\lambda\not=0$. To proceed, consider any triple $(w,u,\lambda)$ satisfying conditions \eqref{EqSOCond1} and \eqref{EqSOCond3} with $\lambda\not=0$. Then $(w,u,\lambda)$ also satisfies \eqref{EqSOCond1a} and \eqref{EqSOCond2a}, and thus there is $\mu\in\R^{l_P}_+$ fulfilling $\gpb^*\lambda=\nabla h(\pb)^*\mu$ and \eqref{EqSOCond3a}. From \cite[Lemma~2.1]{Gfr14b} we deduce that $0=\skalp{\lambda,\nabla\gb(q,u)}$, which implies together with $\gxb^*\lambda=0$ that
\begin{eqnarray*}
0=\skalp{\lambda,\gpb w}=\skalp{\gpb^*\lambda,w}=\skalp{\nabla h(\pb)^*\mu,w}.
\end{eqnarray*}
Furthermore, it follows from Theorem~\ref{PropSetSecOrd}(i) that
\begin{eqnarray*}
\ochi_P\big(\gpb^*\lambda,\pb;w\big)\le-\disp\frac 12\skalp{\nabla^2\skalp{\mu,h}(\pb)w,w}.
\end{eqnarray*}
Applying now Theorem~\ref{PropSetSecOrd}(ii) with $\Omega=Q=C$ and $q(z)=z$ yields $\uchi_C(\lambda,\gb;\nabla\gb(w,u))=0$, and then from \eqref{EqSOCond3a} we obtain the relationships
\begin{eqnarray*}
\lefteqn{\disp\frac 12\skalp{\nabla^2\skalp{\lambda,g}(\pb,\xb)(w,u),(w,u)}+\ochi_P\big(\gpb^*\lambda,\pb;w\big)}\\
&\le&\disp\frac 12\skalp{\nabla^2\skalp{\lambda,g}(\pb,\xb)(w,u),(w,u)}-\frac 12\skalp{\nabla^2\skalp{\mu,h}(\pb)w,w}<0=\uchi_C\big(\lambda,\gb;\nabla\gb(w,u)\big),
\end{eqnarray*}
which show that conditions \eqref{EqSOCond1}--\eqref{EqSOCond3} and $\lambda\not=0$ cannot hold simultaneously. \end{proof}

The following instructive example illustrates the efficiency of the obtained first-order and second-order verification conditions for Robinson stability of a general class of PCS with the splitting structure \eqref{split}. In this example we employ the second-order conditions from Corollary~\ref{CorSecOrdReg} to verify Robinson stability of the system $g_2(p,x)\in C_2$ in \eqref{split} and then deduce Robinson stability of the whole system $g(p,x)\in C$ in \eqref{split} from the first-order Theorem~\ref{ThPartFOReg}.

\begin{Example}[implementation of the verification procedure for Robinson stability]\label{Ex1} {\rm Define the functions $f_i:\R^3\times\R^2\to\R$ for $i=1,2,3$ by 
\begin{eqnarray*}
f_1(p,x)&:=&-x_2-(1/2)x_1^2-p_1+p_2x_2,\\
f_2(p,x)&:=&x_2-(1/2)x_1^2-p_2+p_1x_1,\\
f_3(p,x)&:=&x_1+\vert x_2\vert^{\frac 32}-p_3
\end{eqnarray*}
and consider the system of parameterized nonlinear inequalities $f_i(p,x)\le 0$ with the parameter space
$$
P:=\big\{p\in\R^2\big|\;h(p):=-p_1-p_2+(3/2)p_1^2\le 0\big\}\times\R
$$
and the reference pair $(\pb,\xb)=(0,0)$. This system can be written as a PCS \eqref{EqConstrSystem} with $g=(f_1,f_2,f_3)$ and $C=\R^3_-$. It is convenient to represent \eqref{EqConstrSystem} in the splitting form \eqref{split} with $g_1:=f_3$, $g_2:=(f_1,f_2)$, $C_1=\R_-$, and $C_2=\R^2_-$ for which the results of Theorem~\ref{ThPartFOReg} and Corollary~\ref{CorSecOrdReg} can be applied.

To proceed, consider the mapping $\Tilde g\colon\R^2\times\R^2\to\R^2$ defined by $\Tilde g((p_1,p_2),x):=g_2((p_1,p_2,0),x)$ and the system $\Tilde g(p,x)\in C$ with the parameter space $\Tilde P=h^{-1}(\R_-)$ and the reference pair $(\Tilde p,\xb)=(0,0)$. The mapping $h(\cdot)-\R_-$ is metrically subregular at $(\Tilde p,0)$ since it is actually metrically regular around this point due to the validity of MFCQ therein. It is easy to see furthermore that for every $w\in\R^2$ satisfying $\nabla h(\Tilde p)w=-w_1-w_2\le 0$ the system
$$
\nabla\Tilde g(\tilde p,\xb)(w,u)=\left(\begin{array}
{c}-u_2-w_1\\u_2-w_2
\end{array}\right)\in T_{\R^2_-}\big(\Tilde g(\Tilde p,\xb)\big)=\R^2_-
$$
has the unique solution $u=(0,w_2)$. By Remark~\ref{RemExistenceTangDir}(ii) the conditions in \eqref{EqSOCond1a} and \eqref{EqSOCond2a} amount to
\begin{eqnarray*}
(w_1,w_2,u_1,u_2)\not=(0,0,0,0),\;-w_1-w_2\le 0,\;-u_2-w_1\le 0,\;u_2-w_2\le 0,\\
(0,0)\not=(\lambda_1,\lambda_2)\in\R^2_+,\;\lambda_1(-u_2-w_1)=\lambda_2(u_2-w_2)=0,\;\lambda_1(0,-1)+\lambda_2(0,1)=0
\end{eqnarray*}
yielding  in turn $\lambda_1=\lambda_2>0$ and $u_2=w_2=-w_1$. Hence for every triple $(w,u,\lambda)$ satisfying \eqref{EqSOCond1a} and \eqref{EqSOCond2a} we have $\nabla_p\tilde g(\Tilde p,\xb)^*\lambda=\nabla h(\Tilde p)^*\mu$ together with the equalities $\mu=\lambda_1$ and
\begin{eqnarray*}
\lefteqn{\skalp{\nabla^2\skalp{\lambda,\Tilde g}(\Tilde p,\xb)(w,u),(w,u)}-\skalp{\nabla^2\skalp{\mu,h}(\Tilde p)w,w}}\\
&=&\lambda_1(-u_1^2 +2w_2u_2)+\lambda_2(-u_1^2+2w_1u_1)-3\mu w_1^2
=\lambda_1(-2u_1^2-2w_1u_1-w_1^2)<0.
\end{eqnarray*}
Thus Corollary \ref{CorSecOrdReg} tells us that Robinson stability holds for the system $\Tilde g(p,x)\in\R^2_-$ at $(\Tilde p,\xb)$. Since $g_2((p_1,p_2,p_3),x)=\Tilde g((p_1,p_2),x)$ for all $(p_1,p_2,p_3)\in\R^3$ and $x\in\R^2$, it follows that Robinson stability holds also for the system $g_2(p,x)\in\R^2_-$ at the initial pair $(\pb,\xb)$.

Now we apply Theorem~\ref{ThPartFOReg} with $\zeta(\cdot)=\norm{\cdot}$ to system \eqref{split} splitting above. Since for every $p\in P$ we have $g(p,\xb)-\gb=-(p_1,p_2,p_3)$ and $p_1+p_2\ge\frac 32 p_1^2$, it follows that
\begin{eqnarray*}
\ImDp=\big\{v\in\R^3\big|\;v_1+v_2\ge 0\big\}.
\end{eqnarray*}
Then for every $v\in\ImDp$ the element $u=(v_3,v_2)$ satisfies the conditions
\begin{eqnarray*}
v+\gxb u=(-v_1-v_2,0,0)\in T_C\big(\gb\big),
\end{eqnarray*}
and thus \eqref{EqExistenceTangDir} holds because $C$ is polyhedral; see Remark~\ref{RemExistenceTangDir}(i). By observing that $\gxb^*\lambda=(\lambda_3,\lambda_1+\lambda_2)=0$ yields $\lambda^1=\lambda_3=0$, we deduce from Theorem~\ref{ThPartFOReg} that the Robinson stability property is fulfilled for system \eqref{split} at $(\pb,\xb)$. Note that MFCQ fails to hold for the system $g(\pb,\cdot)\le 0$ at $\xb$.}
\end{Example}

The obtained results on Robinson stability in Theorem~\ref{ThPartSOReg} allow us to derive new {\em second-order} conditions for {\em metric subregularity} of constraint systems. Earlier results of this type have been known only in some particular settings: in the case where the constraint set $C$ is the union of finitely many convex polyhedra \cite{Gfr14b} and for subdifferential systems that can be written in a constraint form \cite{DMN13}.

\begin{Corollary}[second-order conditions for metric subregularity of constraint systems]\label{subreg-second} Consider the constraint system $g(x)\in C$, where $C$ is an arbitrary closed set, and where $g:\R^n\to\R^l$ is twice differentiable at $\xb\in g^{-1}(C)$. Assume that for every pair $(u,\lambda)$ satisfying
\begin{eqnarray*}
\label{EqSOCond1b}&&0\not=u,\;\nabla g(\xb)u\in T_C\big(\gb\big),\\
\label{EqSOCond2b}&&\lambda\in N_C\big(g(\xb);\nabla g(\xb)u\big),\;\nabla g(\xb)u^*\lambda=0,\\
\label{EqSOCond3b}&&\frac 12\skalp{\nabla^2\skalp{\lambda,g}(\xb)u,u}\geq \uchi_C\big(\lambda,g(\xb);\nabla g(\xb)u\big)
\end{eqnarray*}
we have $\lambda=0$. Then the mapping $g(\cdot)-C$ is metrically subregular at $(\xb,0)$.
\end{Corollary}
\begin{proof} Follows directly from Theorem~\ref{ThPartSOReg} with $g(p,x)=g(x)$.\end{proof}

\section{Applications to Parametric Variational Systems}

In this section we first show that the obtained results on Robinson stability of PCS \eqref{EqConstrSystem} allow us to establish new verifiable conditions for {\em robust Lipschitzian stability} of their solution maps \eqref{Gamma}. By the latter we understand, in the case where $P$ is a metric space equipped with the metric $\rho(\cdot,\cdot)$, the validity of the {\em Lipschitz-like} (Aubin, pseudo-Lipschitz) property of $\Gamma\colon P\tto\R^n$ around $(\ox,\op)\in\gph\Gamma$, i.e., the existence of a constant $\ell\ge 0$ and neighborhoods $V$ of $\op$ and $U$ of $\ox$ such that
\begin{eqnarray}\label{lip-l}
\Gamma(p)\cap U\subset\Gamma(p')+\ell\,\rho(p,p')\B_{\R^n}\;\mbox{ for all }\;p,p'\in V.
\end{eqnarray}

Various conditions ensuring the Lipschitz-like property of solution maps as in \eqref{Gamma} have been obtained in  many publications; see, e.g., \cite{Mor06,RoWe98} and the references therein. The result most close to the following theorem is given in \cite[Theorem~4.3]{ChiYaoYen10}, which shows that Robinson stability (called ``Robinson metric regularity" in \cite{ChiYaoYen10}) of \eqref{EqConstrSystem} at $(\op,\ox)$ yields the Lipschitz-like property of \eqref{Gamma} around this point when $P$ is a normed space and some additional assumption on $\Gamma$ is imposed.

\begin{Theorem}[Lipschitz-like property of solution maps to PCS]\label{ThLip} In addition to the standing assumptions for \eqref{EqConstrSystem}, suppose that $P$ is a metric space and that $g$ is locally Lipschitzian near $(\pb,\xb)$. Then Robinson stability of PCS \eqref{EqConstrSystem} implies that the solution map \eqref{Gamma} is Lipschitz-like around this point.
\end{Theorem}
\begin{proof} By Definition~\ref{DefUMSP} there are neighborhoods $V$ of $\pb$ and $U$ of $\xb$ together with $\kappa\ge 0$ such that \eqref{EqUMSP} holds. We can also assume that $g$ is Lipschitz continuous on $V\times U$ with modulus $L\ge 0$ and then, by implication (i)$\Longrightarrow$(ii) of Theorem~\ref{ThEquMSP_BMP} with shrinking $V$ if needed, get that $\Gamma(p)\cap U\not=\emp$ for all $p\in V$. Next consider any elements $p,p'\in V$ and $x\in\Gamma(p)\cap U$. Since $g(p,x)\in C$, it follows that
$$
{\rm dist}\big(g(p',x);C\big)\le\norm{g(p',x)-g(p,x)}\le L\rho(p,p'),
$$
which yields together with \eqref{EqUMSP} the distance estimates
\begin{eqnarray*}
{\rm dist}\big(x;\Gamma(p')\big)\le\kappa{\rm dist}\big(g(p',x);C\big)\le L\kappa\rho(p,p'),
\end{eqnarray*}
or equivalently, $x\in\Gamma(p')+L\kappa\rho(p,p')\B_{\R^n}$. By $x\in\Gamma(p)\cap U$ this tells us that
\begin{eqnarray*}
\Gamma(p)\cap U\subset\Gamma(p')+L\kappa\rho(p,p')\B_{\R^n}\;\mbox{ for all }\;p,p'\in V,
\end{eqnarray*}
which verifies the Lipschitz-like property \eqref{lip-l} of $\Gamma$ around $(\pb,\xb)$ with modulus $\ell=L\rho$.\end{proof}

We apply Theorem~\ref{ThLip} and the efficient conditions for Robinson stability of PCS established above to studying the Lipschitz-like property of solution maps to the {\em parametric variational systems} (PVS):
\begin{eqnarray}\label{EqVarSyst}
0\in F(p,x)+\widehat N_{\Omega(p)}(x)\;\mbox{ for }\;x\in\R^n\;\mbox{ and }\;p\in P\subset\R^m,
\end{eqnarray}
where $F:\R^m\times\R^n\to\R^n$ is continuously differentiable, and where for each $p\in\R^m$ the parameter-dependent set $\Omega(p)$ is defined by the nonlinear inequalities
\begin{eqnarray*}
\Omega(p):=\big\{x\in\R^n\big|\;\varphi_i(p,x)\le 0,\;i=1,\ldots,l_I\big\}
\end{eqnarray*}
described by ${\cal C}^2$-smooth functions $\varphi_i:\R^m\times\R^n\to\R$. According to the terminology of variational analysis \cite{Mor06,RoWe98}, by {\em variational} systems we understand generalized equations of type \eqref{EqVarSyst} the multivalued parts of which are given by subdifferential/normal cone mappings. In case \eqref{EqVarSyst} significant difficulties arise from the parameter dependence of $\O(p)$. When the sets $\O(p)$ are {\em convex}, PVS \eqref{EqVarSyst} relate to {\em quasi-variational inequalities} (where $\O$ may also depend on $x$) the Lipschitz-like property of which has been studied in \cite{MorOut07} on the basis of coderivative analysis (via the Mordukhovich criterion) and coderivative calculus rules. Here we don't assume the convexity of $\O(p)$ and conduct our sensitivity analysis via Theorem~\ref{ThLip} and the obtained conditions for Robinson stability.

It is well known that mild qualification conditions at $x\in\Omega(p)$ as used below ensure that
\begin{eqnarray*}
\widehat N_{\Omega(p)}(x)=\nabla_x\varphi(p,x)^*N_{\R^{l_I}_-}\big(\varphi(p,x)\big),
\end{eqnarray*}
where $\ph=(\ph_1,\ldots,\ph_{l_I})$. This allows us to replace \eqref{EqVarSyst} by the following constraint system:
\begin{eqnarray}\label{EqEnhSyst1}
\left\{\begin{array}{ll}
F(p,x)+\nabla_x\varphi(p,x)^*y=0,\\
\big(\varphi(p,x),y\big)\in\Gr N_{\R^{l_I}_-},
\end{array}\right.
\end{eqnarray}
which can be written in the PCS form \eqref{EqConstrSystem} as
\begin{eqnarray}\label{EqSystKKT}
g(p,x,y):=\left(\begin{array}{ll}G(p,x,y)\\\big(\varphi(p,x),y\big)\end{array}\right)\in C:=\{0_{\R^n}\}\times\Gr  N_{\R^{l_I}_-},
\end{eqnarray}
where $G(p,x,y):=F(p,x)+\nabla_x\varphi(p,x)^*y$ is the {\em Lagrangian} associated with the parametric variational system under consideration. The next theorem justifies the Lipschitz-like property of
\begin{eqnarray}\label{EqSolMapKKT}
\Gamma\colon P\tto\R^n\times\R^{l_I}\;\mbox{ with }\;\Gamma(p):=\big\{(x,y)\in\R^n\times\R^{l_I}\big|\;(p,x,y)\;\mbox{ satisfies }\;\eqref{EqEnhSyst1}\big\}
\end{eqnarray}
near the given reference point $(\pb,\xb,\yb)\in\Gr\Gamma$.

\begin{Theorem}[Robinson stability and Lipschitz-like properties of parametric variational systems]\label{ThPartRegKKT} Suppose that for every $w\in T_P(\pb)$ there exists a pair $(u,v)\in\R^n\times\R^{l_I}$ satisfying the conditions
\begin{eqnarray}\label{EqLinKKT1}
\left\{\begin{array}{ll}
\nabla_p G(\pb,\xb,\yb)w+\nabla_x G(\pb,\xb,\yb)u+\nabla_x\varphi(\pb,\xb)^*v=0,\\
(\nabla\varphi(\pb,\xb)(w,u),v)\in T_{\Gr N_{\R^{l_I}_-}}\big((\varphi(\pb,\xb),\yb\big).
\end{array}\right.
\end{eqnarray}
Assume in addition that for every triple $(q,u,v)\not=(0,0,0)$ fulfilling \eqref{EqLinKKT1} and for every triple of multipliers $(z,\lambda,d)\in\R^n\times\R^{l_I}\times\R^{l_I}$ we have the implication
\begin{eqnarray}\label{EqCoDerivKKT}
\left.\begin{array}{ll}
\nabla_x G(\pb,\xb,\yb)^*z+\nabla_x\varphi(\pb,\xb)^*\lambda=0,\\
\nabla_x\varphi(\pb,\xb)z+d=0,\\
(\lambda,d)\in N_{\Gr N_{\R^{l_I}_-}}\Big(\big(\varphi(\pb,\xb),\yb\big);\big(\nabla\varphi(\pb,\xb)(w,u),v\big)\Big)\end{array}\right\}
\Longrightarrow(z,\lambda)=(0,0).
\end{eqnarray}
Then Robinson stability holds at $(\pb,\xb,\yb)$ for the system $g(p,x,y)\in C$ given by \eqref{EqSystKKT}  and the solution map $\Gamma$ given by \eqref{EqSolMapKKT}
is Lipschitz-like around this triple.
\end{Theorem}
\begin{proof} Note that the implication in \eqref{EqCoDerivKKT} ensures that $d=0$ as well, since it automatically follows from $z=0$. Then both statements of the theorem follow from the combination of the results presented in Theorem~\ref{ThLip}, Corollary~\ref{CorPartFOREg}, and Remark~\ref{RemExistenceTangDir}(i).\end{proof}

To proceed with applications of Theorem~\ref{ThPartRegKKT}, deduce from \cite[Lemma~1]{GfrKl15} that for all the pairs $(\varphi,y)\in\Gr N_{\R^{l_I}_-}$ and $(s,v)\in T_{\Gr N_{\R^{l_I}_-}}(\varphi,y)$ we have the componentwise conditions
\begin{eqnarray*}
\begin{array}{ll}
(s_i,v_i)\in T_{\Gr N_{\R_-}}(\varphi_i,y_i)\;\mbox{ for }\;i=1,\ldots,l_I,\quad{and}\\
N_{\Gr N_{\R^{l_I}_-}}\big((\varphi,y);(s,v)\big)=\disp\prod_{i=1}^{l_I}N_{\Gr N_{\R_-}}\big((\varphi_i,y_i);(s_i,v_i)\big).
\end{array}
\end{eqnarray*}
Straightforward calculations show that for every $(\varphi,y)\in\Gr N_{\R_-}$ the following hold:
\begin{eqnarray*}
T_{\Gr N_{\R_-}}(\varphi,y)&=&
\left\{(s,v)\in\R^2\Big|\;\begin{cases}s=0&\mbox{if $\varphi=0<y$},\\
v=0&\mbox{if $\varphi<0=y$},\\
s\le 0\le v,\;sv=0&\mbox{if $\varphi=0=y$}
\end{cases}\right\},\\
\widehat N_{\Gr N_{\R_-}}(\varphi,y)&=&
\left\{(\lambda,d)\in\R^2\Big|\;\begin{cases} d=0&\mbox{if $\varphi=0<y$},\\
\lambda=0&\mbox{if $\varphi<0=y$},\\
\lambda\ge 0\ge\ d&\mbox{if $\varphi=0=y$}
\end{cases}\right\},\\
N_{\Gr N_{\R_-}}(\varphi,y)&=&\begin{cases}
\widehat N_{\Gr N_{\R_-}}(\varphi,y)&\mbox{if $(\varphi,y)\not=(0,0)$},\\
\big\{(\lambda,d)\in\R^2\big|\;\lambda>0>d\mbox{ or }\lambda d=0\big\}&\mbox{if $(\varphi,y)=(0,0).$}
\end{cases}
\end{eqnarray*}
This ensures that for every $(s,v)\in T_{\Gr N_{\R_-}}(\varphi,y)$ we have the representation
\begin{eqnarray*}
N_{\Gr N_{\R_-}}\big((\varphi,y);(s,v)\big)=\begin{cases}N_{\Gr N_{\R_-}}(\varphi,y)&\mbox{if $(\varphi,y)\not=(0,0)$,}\\
N_{\Gr N_{\R_-}}(s,v)&\mbox{if $(\varphi,y)=(0,0)$.}
\end{cases}
\end{eqnarray*}

Finally in this section, we illustrate in detail the procedure of applications of Theorem~\ref{ThPartRegKKT} and the previous Robinson stability results on which this theorem is based to determine the validity of the Lipschitz-like property of parametric variational systems arising from the Karush-Kuhn-Tucker (KKT) optimality conditions for parameterized nonlinear programs.

\begin{Example}[robust Lipschitzian stability of KKT systems]\label{robust-kkt} {\rm Consider the following mathematical program depending on the parameter vector $p\in\R^2$:
\begin{eqnarray}\label{nlp}
\left\{\begin{array}{ll}
\mbox{minimize }\;f(p,x):=\disp\frac 12 x_1^2-\frac 12 x_2^2-\skalp{p,x}\;\mbox{ over }\;x\in\R^2\\
\mbox{subject to }\;\varphi_1(x):= \disp-\frac 12 x_1+x_2\le 0,\quad\varphi_2(x):=-\frac 12 x_1-x_2\le 0.
\end{array}\right.
\end{eqnarray}
It is easy to check that $\xb=0$ is a local minimizer of \eqref{nlp} for $\op=0$. The KKT conditions for \eqref{nlp} whenever $p\in\R^2$ can be written in the variational form \eqref{EqEnhSyst1} with $F(p,x):=\nabla_x f(p,x)^*$. Furthermore, we can represent these conditions as the parametric constraint system \eqref{EqSystKKT} with
$$
G(p,x,y):=\left(\begin{array}
{c}x_1-p_1-\frac 12 y_1-\frac 12 y_2\\
-x_2-p_2+y_1-y_2
\end{array}\right).
$$
It is not hard to observe that for $\pb=(0,0)$ the point $\xb=(0,0)$ together with the multiplier $\yb=(0,0)$ is the unique solution to the KKT system under consideration. The reader can check that the previously known conditions for the Lipschitz-like property of the solution map to the KKT system based on metric regularity are not able to clarify the validity of this property around the given trivial solution.

We are going to use for this purpose the new results established in Theorem~\ref{ThPartRegKKT}. To this end, let us show that for every subset $P\subset\R^2$ such that $(0,0)=\pb\in P$ and the contingent cone $T_P(\pb)$ doesn't contain the directions $(1,1/2)$ and $(1,-1/2)$, the constraint system $g(p,x,y)\in C$ defined by \eqref{EqSystKKT} enjoys the Robinson stability property at $(\pb,\xb,\yb)$ and thus the corresponding solution map $\Gamma\colon P\tto\R^2\times\R^2$ from \eqref{EqSolMapKKT} is Lipschitz-like around this point.

To proceed, we need to verify first the validity of all the assumptions of Theorem~\ref{ThPartRegKKT}. Note that conditions \eqref{EqLinKKT1} amount in our case to saying that
\begin{eqnarray*}
\left(\begin{array}
{c}u_1-w_1-\frac 12 v_1-\frac 12 v_2\\
-u_2-w_2+v_1-v_2
\end{array}\right)=\left(\begin{array}{c}0\\0\end{array}\right),\\
-\frac 12 u_1+u_2\le 0,\;v_1\ge 0,\;v_1\Big(-\frac 12 u_1+u_2\Big)=0,\\
-\frac 12 u_1-u_2\le 0,\;v_2\ge 0,\;v_2\Big(-\frac 12 u_1-u_2\Big)=0.
\end{eqnarray*}
Further, it follows that $(w,u,v)$ solves this system if and only if we have $(w,u,v)\in\Gr\Tilde\Gamma$, where the mapping $\Tilde\Gamma\colon\R^2\tto\R^2\times\R^2$ is defined by
\begin{equation}\label{EqExGamma}\Tilde\Gamma(w)=\begin{cases}
\big\{z_1(w)\big\}&\mbox{if $w\in Q_1:=\big\{w\big|-\frac 12 w_1+w_2> 0,\;2w_1+w_2\ge 0\big\}$,}\\
\big\{z_1(w),z_2(w),z_3(w)\big\}&\mbox{if $w\in Q_2:=\big\{w\big|\;-\frac 12 w_1+w_2\le 0,\;-\frac 12 w_1-w_2\le 0\big\}$,}\\
\big\{z_3(w)\big\}&\mbox{if $w\in Q_3:=\big\{w\big|\;-\frac 12 w_1-w_2>0,\;2w_1-w_2\ge 0\big\}$,}\\
\big\{z_4(w)\big\}&\mbox{if $w\in Q_4:=\big\{w\big|\;2w_1+w_2\le 0,\;2w_1-w_2\le 0\big\}$,}\\
\end{cases}
\end{equation}
and where the functions $z_i(w)$, $i=1,\ldots,4$, in \eqref{ThPartRegKKT} are specified as follows:
\begin{eqnarray*}
z_1(w)&:=&\Big(\Big(\frac 43 w_1+\frac 23 w_2,\frac 23 w_1+\frac 13 w_2\Big),\Big(\frac 23 w_1+\frac 43 w_2,0\Big)\Big),\;z_2(w):=\big((w_1,-w_2),(0,0)\big),\\
z_3(w)&:=&\Big(\Big(\frac 43 w_1-\frac 23 w_2,-\frac 23 w_1+\frac 13 w_2\Big),\Big(0,\frac 23 w_1-\frac 43 w_2\Big)\Big),\\
z_4(w)&:=&\Big((0,0),\Big(-w_1+\frac{w_2}2,-w_1-\frac{w_2}2\Big)\Big).
\end{eqnarray*}
Hence for every $w\in\R^2$ there exists a solution $(u,v)$ to this system. Now consider such a solution to \eqref{EqExGamma} that $(0,0,0)\not=(w,u,v)$ with $w\in T_P(\pb)$ and then take a triple $(z,\lambda,d)$ from the left-hand side of \eqref{EqCoDerivKKT}. Thus the implication in \eqref{EqCoDerivKKT} says that
\begin{eqnarray}\nonumber
&&z_1-\frac 12\lambda_1-\frac 12\lambda_2=0,\;-z_2+\lambda_1-\lambda_2=0,\;-\frac 12 z_1+z_2+d_1=0,\;-\frac 12 z_1-z_2+d_2=0,\\
\label{EqKKTAux5}
&&(\lambda_1,d_1)\in N_{\Gr N_{\R_-}}\Big((0,0);\Big(-\frac 12 u_1+u_2,v_1\Big)\Big)=N_{\Gr N_{\R_-}}\Big(-\frac 12 u_1+u_2,v_1\Big),\\
\label{EqKKTAux6}
&&(\lambda_2,d_2)\in N_{\Gr N_{\R_-}}\Big((0,0);\Big(-\frac 12 u_1-u_2, v_2\Big)\Big)=N_{\Gr N_{\R_-}}\Big(-\frac 12 u_1-u_2, v_2\Big).
\end{eqnarray}
By eliminating $z_1,z_2$ we deduce from the relationships above that
\begin{eqnarray}\label{EqKKTAux1a}
d_1=-\frac 34\lambda_1+\frac 54\lambda_2,\quad d_2=\frac 54 \lambda_1-\frac 34\lambda_2
\end{eqnarray}
and easily get from $(0,0,0)\not=(w,u,v)$ that $w\not=0$. Thus the following four cases should be analyzed:

{\bf (i)} $w\in Q_1\cup Q_2:=\{w|\;-\frac 12 w_1-w_2\le 0,\;2w_1+w_2\ge 0\}$ and $(u,v)=z_1(w)$. Since $(1,-\frac 12)\not\in T_P(\pb)$ in this case, it tells us that
$-\frac 12 w_1-w_2< 0$ and therefore
$$
-\frac 12 u_1+u_2=0,\;v_1=\frac 43\Big(\frac 12 w_1+w_2\Big)>0,\;-\frac 12 u_1-u_2=-u_1=-\frac 23\Big(2w_1+w_2\Big)\le 0,\;v_2=0.
$$
From \eqref{EqKKTAux5} we have that $d_1=0$, and so $\lambda_1=\frac 53\lambda_2$ and $d_2=\frac 43\lambda_2$. It follows that $d=\lambda=(0,0)$ is the only solution to system \eqref{EqKKTAux1a} if $\lambda_2d_2=0$. By \eqref{EqKKTAux6} the case $\lambda_2d_2\not=0$ could only be possible if $-\frac 12 u_1-u_2=-u_1=-\frac 23(2w_1+w_2)\leq 0$. This results in $\lambda_2>0>d_2$ contradicting $d_2=\frac 43 \lambda_2$. Hence $d=\lambda=0$ is the only pair satisfying \eqref{EqKKTAux5}, \eqref{EqKKTAux6}, and \eqref{EqKKTAux1a}; thus we arrive at $z=0$.

{\bf (ii)} $w\in Q_2\cup Q_3=\{w|\;-\frac 12 w_1+w_2\le 0,\;2w_1-w_2\ge 0\}$ and $(u,v)=z_3(q)$. Using in this case the same arguments as in (i) gives us $(z,\lambda,d)=(0,0,0)$.

{\bf (iii)} $w\in Q_2$ and $(u,v)=z_2(w)$. Since $(1,\pm\frac 12)\not\in T_P(\pb)$ in this case, we have $-\frac 12 w_1<w_2<\frac 12 w_1$, and thus it follows from \eqref{EqKKTAux5} and \eqref{EqKKTAux6} that $\lambda_1=\lambda_2=0$. Thus we also get $d=0$ and $z=0$.

{\bf (iv)} $w\in Q_4$ and $(u,v)=z_4(w)$. In this case we have $v_1=-w_1+\frac{w_2}2\ge 0$ and $v_2=-w_1-\frac{w_2}2\ge 0$, but these two values can't be zero simultaneously due to $w\not=0$. If both values $v_1,v_2$ are positive, then $d_1=d_2=0$ and consequently $\lambda=z=0$. If only one value, say $v_1$, is positive, then $d_1=0$, and so we get from from \eqref{EqKKTAux1a} that $\lambda_1=\frac 53\lambda_2$ and $d_2=\frac 43\lambda_2$. It follows from \eqref{EqKKTAux6} that either $\lambda_2d_2=0$ or $\lambda_2>0>d_2$ implying thus that $\lambda_2=d_2=0$ and consequently $\lambda_1=0$ and $z=0$.

As shown, in each of these cases we have $(z,\lambda,d)=0$, and hence \eqref{EqCoDerivKKT} holds. Therefore Robinson stability of the system $g(p,x,y)\in C$ under consideration at $(\pb,\xb,\yb)$ and the Lipschitz-like property of tits solution map $\Gamma$ around this point follow from Theorem~\ref{ThPartRegKKT}.

Observe further that the solution map $\Gamma$ to the KKT system for \eqref{nlp} is actually the restriction of $\Tilde\Gamma$ from \eqref{EqExGamma} on $P$. If $z_2(p)\in\Gamma(p)$, the $x$-part of it constitutes a stationary solution, while those of $z_1(p)$ and $z_3(p)$ are local minimizers for \eqref{nlp} provided that $z_1(p)\in\Gamma(p)$ and $z_3(p)\in\Gamma(p)$ but excepting the cases of $0\le p_1=2p_2$ and $0\le p_1=-2p_2$ when $z_2=z_1$ and $z_2=z_3$, respectively. The $x$-part of $z_4(p)$ is a local minimizer whenever $z_4(p)\in\Gamma(p)$.

Now consider two points $p',p''\in P$ near $\pb$ with $p'\in Q_1\cap P$ and $p''\in Q_2\cap P$. It follows that $\Gamma(p')=\{z_1(p')\}$ and $\Gamma(p'')=\{z_1(p''),z_2(p''),z_3(p'')\}$. Furthermore, we have $\norm{z_1(p')-z_1(p'')}\le L\norm{p'-p''}$, but for $z_2(p'')$ and $z_3(p'')$ only some bounds of the form
\begin{eqnarray*}
{\rm dist}\big(z_i(p'');\Gamma(p')\big)=\norm{z_i(p'')-z_1(p')}\le L(\norm{p'-\pb}+\norm{p''-\pb}),\;i=2,3,
\end{eqnarray*}
are available, where $L$ is sufficiently large. However, the condition $(1,\frac 12)\not\in T_P(\pb)$ ensures the existence of $\epsilon_1>0$ such that the lower estimate
\begin{eqnarray*}
\norm{p'-p''}\ge\epsilon_1(\norm{p'-\pb}+\norm{p''-\pb})
\end{eqnarray*}
holds for all $p'\in Q_1\cap P$ and $p''\in Q_2\cap P$ close to $\pb$, which results in the inclusion
\begin{eqnarray*}
\Gamma(p'')\subset\Gamma(p')+\frac L{\epsilon_1}\norm{p''-p'}\B_{\R^4}.
\end{eqnarray*}
Similar considerations apply to a pair $(p'',p''')\in(Q_2\cap P)\times(Q_3\cap P)$. Due to $(1,-\frac 12)\not\in T_P(\pb)$ we get
\begin{eqnarray*}
\norm{p'''-p''}\ge\epsilon_2(\norm{p'''-\pb}+\norm{p''-\pb})
\end{eqnarray*}
with $\epsilon_2>0$ for $p'',p'''$ near $\pb$, which implies in turn the inclusion
\begin{eqnarray*}
\Gamma(p'')\subset\Gamma(p''')+\frac L{\epsilon_2}\norm{p''-p'''}\B_{\R^4}.
\end{eqnarray*}
Summarizing our consideration shows that the solution map $\Gamma$ to the variational KKT system associated with \eqref{nlp} is Lipschitz-like around the reference point for every subset $P\subset\R^2$ described above.}
\end{Example}

\section{Concluding Remarks}

This paper studies a well-posedness property of general parametric constraint systems \eqref{EqConstrSystem}, which goes back to Robinson and is named here {\em Robinson stability}. We conduct a rather detailed analysis of this fundamental property with deriving verifiable first-order and second-order conditions for its validity by using advanced tools of variational analysis and generalized differentiation. As consequences of the main results, new conditions for the Lipschitz-like/Aubin property for solution maps to constraint and certain classes of variational systems are derived and illustrated by nontrivial examples.

As discussed in Section~1, the name ``Robinson metric regularity" used for the underlying property in some publications seems to be misleading, since this property doesn't correspond to the conventional understanding of metric regularity. On the other hand, we employ a useful interpretation of Robinson stability for \eqref{EqConstrSystem} as the {\em uniform} metric {\em subregularity} of the mapping $g(p,\cdot)-C$ over the given class of parameter perturbations. This approach leads us, in particular, to establishing new sufficient conditions of metric subregularity for nonpolyhedral constraint systems.

Robinson stability and related topics are planned to be a focus of our future research, in both theoretical and applied frameworks. Among them we mention a detailed investigation of Robinson stability for parametric variational systems, with the specific emphasis on variational inequalities and nonlinear complementarity. Another important topic of our particular concentrations is {\em full stability} of local minimizers for various constrained optimization problems including NLPs, conic programming, bilevel optimization, etc. In this direction, which has been in fact our original motivation for the current study, we plan to investigate both {\em Lipschitzian} (as in \cite{LPR00}) and {\em H\"olderian} (as in \cite{MN14,MNR15}) notions of full stability and to obtain results at the same level of perfection as in our recent study of tilt stability (a special case of full stability) for NLPs in \cite{GM15}. 

{\bf Acknowledgements.} The research of the first author was partially supported by the Austrian Science Fund (FWF) under grants P26132-N25 and P29190-N32. The research of the second author was partially supported by the USA National Science Foundation under grants DMS-12092508 and DMS-1512846, by the USA Air Force Office of Scientific Research under grant No.\,15RT0462, and  by the Ministry of Education and Science of the Russian Federation (the Agreement No. 02.a03.21.0008 of 24.06.2016). The authors gratefully acknowledge useful remarks by two anonymous referees as well as Aram Arutyunov, Alex Kruger, and Diethard Klatte that allowed us to improve the original presentation.

\end{document}